\documentclass[english, 10pt,reqno]{amsart}
\usepackage{geometry}  
\usepackage{amssymb,amsmath,amsthm,amsfonts,color}
\usepackage{mathrsfs,dsfont, comment,mathscinet}
\usepackage{graphicx}
\usepackage{epstopdf}
\usepackage{mathtools}
\usepackage{babel}
\usepackage{enumerate,esint}

\usepackage{indentfirst}
\usepackage{bm}
\usepackage{picinpar}
\usepackage{caption}
\usepackage{subcaption}
\usepackage{soul}
\usepackage[normalem]{ulem}

\usepackage[colorlinks=true, pdfstartview=FitV, linkcolor=blue, citecolor=blue, urlcolor=blue]{hyperref}

\usepackage[toc,page]{appendix}

\usepackage{etoolbox}
\usepackage{authblk}
\usepackage{amsaddr}

\usepackage[running]{lineno}
\newcommand*\patchAmsMathEnvironmentForLineno[1]{%
  \expandafter\let\csname old#1\expandafter\endcsname\csname #1\endcsname
  \expandafter\let\csname oldend#1\expandafter\endcsname\csname end#1\endcsname
  \renewenvironment{#1}%
     {\linenomath\csname old#1\endcsname}%
     {\csname oldend#1\endcsname\endlinenomath}}% 
\newcommand*\patchBothAmsMathEnvironmentsForLineno[1]{%
  \patchAmsMathEnvironmentForLineno{#1}%
  \patchAmsMathEnvironmentForLineno{#1*}}%
\AtBeginDocument{%
\patchBothAmsMathEnvironmentsForLineno{equation}%
\patchBothAmsMathEnvironmentsForLineno{align}%
\patchBothAmsMathEnvironmentsForLineno{flalign}%
\patchBothAmsMathEnvironmentsForLineno{alignat}%
\patchBothAmsMathEnvironmentsForLineno{gather}%
\patchBothAmsMathEnvironmentsForLineno{multline}%
}

\graphicspath{ {images/} }

\allowdisplaybreaks %This command allows 

\mathtoolsset{showonlyrefs} %This command will only number the equation which is referred later.

\makeatletter

\usepackage{fancyhdr}
 
\makeatletter
\patchcmd{\@maketitle}
  {\ifx\@empty\@dedicatory}
  {\ifx\@empty\@date \else {\vskip3ex \centering\footnotesize\@date\par\vskip1ex}\fi
   \ifx\@empty\@dedicatory}
  {}{}
\patchcmd{\@maketitle}
  {\ifx\@empty\@date\else \@footnotetext{\@setdate}\fi}
  {}{}{}
\makeatother
\begingroup
\newtheorem{theorem}{Theorem}[section]
\newtheorem{lemma}[theorem]{Lemma}
\newtheorem{proposition}[theorem]{Proposition}
\newtheorem{corollary}[theorem]{Corollary}

\endgroup

\theoremstyle{definition}
\begingroup
\newtheorem{define}[theorem]{Definition}
\newtheorem{remark}[theorem]{Remark}

\endgroup
\newcommand\ba[1]{\begin{align}\label{#1}}
\newcommand\ea{\end{align}}
\newcommand\bas{\begin{align*}}
\newcommand\eas{\end{align*}}

\newcommand\ee{\end{equation}}
\newcommand\be{\begin{equation}}
\newcommand\ees{\end{equation*}}
\newcommand\bes{\begin{equation*}}
\mathchardef\emptyset="001F

\newcommand{\e}{\varepsilon}

\newcommand{\Om}{\Omega}

\newcommand{\R}{{\mathbb R}}

\newcommand{\rn}{{{\R}^N}}

\newcommand{\wto}{\rightharpoonup}
\newcommand{\wtos}{\mathrel{\mathop{\rightharpoonup}\limits^*}}

\newcommand{\wtogs}{\mathrel{\mathop{\to}\limits^{\text {ss}}}}
\newcommand{\wtogsp}{\mathrel{\mathop{\to}\limits^{\text{ss-}(p)}}}

\newcommand{\N}{{\mathbb{N}}}

\newcommand\norm[1]{\left\|#1\right\|}

\newcommand{\abs}[1]{\left\lvert#1\right\rvert} 
\newcommand{\fsp}[1]{\left(#1\right)} 
\newcommand{\fmp}[1]{\left[#1\right]}
\newcommand{\flp}[1]{\left\{#1\right\}}

\newcommand{\vp}{\varphi}

\newcommand{\limn}{\lim_{n\rightarrow\infty}}

\newcommand{\lime}{\lim_{\e\rightarrow0}}

\newcommand{\divg}{{\operatorname{div}}}

\newcommand{\seqn}[1]{\left\{#1\right\}_{n=1}^\infty}

\newcommand{\liminfn}{{\liminf_{n\to\infty}}}
\newcommand{\ir}{{\lfloor r\rfloor}}
\newcommand{\IM}{{\operatorname{IM}}}

\newcommand\M{\mathbb M}

\definecolor{CMUred}{RGB}{153,0,0}
\definecolor{CMUgreen}{RGB}{0,135,81}
\definecolor{CMUblue}{RGB}{0,51,127}
\definecolor{Pblue}{RGB}{87,158,208}
\newcommand{\argmin}{{\operatorname{arg\,min}}}

\newcommand{\hnmo}{{\mathcal H^{N-1}}}

\newcommand\mb{\mathcal{M}_b}

\newcommand{\ellp}{{\ell^p}}

\def\argmin{\mathop{\rm arg\, min}}

\DeclareRobustCommand{\Om}{\Omega}

\numberwithin{equation}{section}

\newcommand{\normmm}[1]{{\left\vert\kern-0.25ex\left\vert\kern-0.25ex\left\vert #1 
    \right\vert\kern-0.25ex\right\vert\kern-0.25ex\right\vert}}
\newcommand{\BLK}{\color{black}}

\title{Real order (an)-isotropic total variation in image processing - Part I: Analytical Analysis and Functional properties}

\author[P. Liu] {Pan Liu}
 \address[Pan Liu]{Centre of Mathematical Imaging and Healthcare,\\ 
Department of Pure Mathematics and Mathematical Statistics, \\
 University of Cambridge,\\
  Wilberforce Road, Cambridge CB3 0WA, UK}
 \email[P. Liu] {panliu.0923@maths.cam.ac.uk}

 \author[X.Y. Lu] {Xin Yang Lu}
 \address[Xin Yang Lu]{Department of Mathematical Sciences, Lakehead University,
 955 Oliver Road, Thunder Bay, ON, Canada\\
 AND\\
 Department of Mathematics and Statistics
Burnside Hall, McGill University\\805 Sherbrooke Street West, Montreal, QC, Canada }
 \email[X. Lu] {xlu8@lakeheadu.ca}


\date{\today}                                           % Activate to display a given date or no date

\begin{document}

%\linenumbers
    %\lipsum

\begin{abstract}
In this paper, a variational, multi-dimensional model for image reconstruction is proposed, in which the regularization term
 consists of the $r$-order (an)-isotropic total variation seminorms $TV^r$,
with $r\in \R^+$, defined via the Riemann-Liouville fractional derivative. Key properties, such as the lower semi-continuity and compactness with respect to both the function
and the order of derivative $r$, are studied. 
This paper, the first of our series of works on analytical and numerical aspects of the model, as well as the learning of optimal order $r$ for particular imaging tasks, 
provides a comprehensive analysis of the behavior of $TV^r$ in the space of functions with bounded (fractional order) total variation.

\end{abstract}

\maketitle
\tableofcontents

\thispagestyle{empty}%this command remove the page number at the title page

%\section{Problems to fix}
%We are listing a series of questions below:
%\begin{enumerate}[1.]
%\item the Proposition \ref{compact_seminorm} should take first priority now. It is the key point to our new training scheme. I add a small idea of prove.
%\item the new scheme does not require of the optimizing order. i.e., we do not seek for $\alpha$ first and $s$ later, but more or less at the same time.
%
%
%\end{enumerate}
\section{Introduction}\label{sec:intro}
Methods combining Semi-supervised Learning (\emph{SSL}) approaches and variational models have recently received increasing attention in image processing and inverse problems. 
The \emph{SSL} approaches are based on the given prior knowledge of the problem in terms of a training set. In image denoising, for example, 
such prior knowledge might be a pair images: a corrupted image $u_\eta$, and the corresponding clean image $u_c$. 
{\em SSL} approaches 
produce optimal recovery models, by minimizing an assessment function based on the underlying variational problems.\\\\
An example of such learning schemes in image denoising is a bilevel training scheme 
(see \cite{domke2012generic, domke2013learning, tappen2007utilizing,tappen2007learning}) coupled with the celebrated \emph{ROF} imaging denosing model (\cite{rudin1992nonlinear}) as follows.
\begin{flalign}
\text{Level 1. }&\,\,\,\,\,\,\,\,\,\,\,\,\,\,\,\,\,\,\,\,\,\,\,\,\,\,\,\,\,\,\,\,\,\,\,\,\,\tilde\alpha\in\argmin\flp{\norm{u_\alpha-u_c}_{L^2(Q)}^2:\,\,\alpha\in\R^+}\tag{$\mathcal B$-L1}\label{intro_B_train_level1}&\\
\text{Level 2. }&\,\,\,\,\,\,\,\,\,\,\,\,\,\,\,\,\,\,\,\,\,\,\,\,\,\,\,\,\,\,\,\,\,\,\,\,\,u_{\alpha}:=\argmin\flp{\norm{u-u_\eta}_{L^2(Q)}^2+\alpha TV(u):\,\,u\in BV(Q)}.\tag{$\mathcal B$-L2}\label{intro_B_train_level2}&
\end{flalign}
We will refer to it as scheme $\mathcal B$.
In the \emph{ROF} model (embedded in \eqref{intro_B_train_level2}) we denoted by $Q=(0,1)^N$ the domain of the image, $\alpha\in\mathbb R^+$ the \emph{intensity parameter}, and $TV(u)$ 
 the \emph{total variation} of $u$ (also known as \emph{regularizer} in image processing problems). In the \emph{ROF} problem alone, without additional supervision, 
  choosing $\alpha$ too large would often result in the loss of relevenat image details; on the other hand, 
  choosing $\alpha$ too small would produce an image with too much noise. Such issue is attacked by using the \emph{SSL} scheme $\mathcal B$, 
  in which the optimal intensity parameter $\tilde\alpha$ from Level 1 is obtained by minimizing the assessment function, plus the $L^2$-distance from the clean image $u_c$.\\\\
However, it is well known that the \emph{ROF} model in \eqref{intro_B_train_level2} suffers from several drawbacks: 
particularly relevant is the stair-casing effect, and unfortunately, the optimal image $u_{\tilde\alpha}$ reconstructed from scheme $\mathcal B$ inherits such defects. 
One choice to mitigate such drawback is to use a fractional order total variation $TV^r$, in which $r\in\R^+$ represents the order of derivative (when $r=1$ we 
get the standard $TV$ used in \eqref{intro_B_train_level2}). The definition of fractional order total variation will be presented
in the next section. We first present an overview of the current state of art.\\\\
There are several types of point-wise defined non-integer order derivative, such as the \emph{Riemann-Liouville} (RL) derivative, \emph{Caputo} derivative, \emph{Marchaud} derivative, 
\emph{Gr\"unwald-Letnikov} (GL) derivative (see for instance \cite{MR1219954,podlubny2000matrix}). By combining such non-integer order derivative and the standard
total variation $TV$, we introduce the the total $r$ order variation-based model (see e.g., \cite{chen2013fractional, zhang2015total,MR3684877}). This
also defines the following imaging denosing model: minimize
\be\label{mumfordshahori_frac}
 ROF^r(u):=\norm{u-u_\eta}_{L^2(Q)}^2+\alpha TV^r(u), 
\ee
among 
\be\label{BV_too_large}
u\in BV^r(Q):=\flp{u\in L^1(Q):\,\, TV^r(u)<+\infty}.
\ee
The existence of minimizers of \eqref{mumfordshahori_frac}, for fixed $r\in(1,2)$ and $\alpha\in\R^+$, has been proven in \cite{zhang2015total}, along with numerical algorithms to solve such minimizer. \\\\
Apart from \cite{liu2016weightedreg}, all existing approaches for fractional order imaging processing models are formulated and analyzed with
fixed derivative order $r\in(1,2)$, and the main goal is to study the minimizing problem \eqref{mumfordshahori_frac}. However, numerical simulations show that for different images, 
different derivative orders might give different results. Thus, it is relevant to study how to obtain the optimal order $r$ to get the best reconstructed image.\\\\
This paper is the first of our series of works in imaging processing models, where both the optimal order $r$ and intensity parameter $\alpha$ are subject to optimization.
 We include the order $r$ into the learning scheme \eqref{intro_B_train_level1}-\eqref{intro_B_train_level2}, and introduce the following new scheme, which we denote by $\mathcal T$.
\begin{flalign}
\text{Level 1. }&\,\,\,\,\,\,\,\,\,\,\,\,\,\,\,\,\,\,\,\,\,\,\,\,(\tilde \alpha, \tilde r)\in\argmin\flp{\norm{u_{\alpha,r}-u_c}:\,\,{(\alpha,r)\in \R^+\times \R^+}}\tag{$\mathcal T$-L1}\label{S_scheme_L1_intro}&\\
\text{Level 2. }&\,\,\,\,\,\,\,\,\,\,\,\,\,\,\,\,\,\,\,\,\,\,\,\,u_{\alpha,r}:=\argmin\flp{\norm{u-u_\eta}_{L^2(Q)}^2+\alpha  TV^r(u):\,\,{u\in L^1(Q)}}.\tag{$\mathcal T$-L2}\label{S_scheme_L2_intro}&
\end{flalign}
The new scheme $\mathcal T$ allows us to simultaneously optimize both the intensity parameter $\alpha$ and the order $r$ of derivative. Moreover, 
since we shall study the properties of sequences $\seqn{TV^{r}}$, where each $r$ could be either fractional or integer, 
we will refer to our new functional $TV^r$ as the \emph{real order} total variation seni-norm. Thus $TV^r$ can be considered as an extension of the classic
total variation semi-norm, which corresponds to the case $r=1$.\\\\
We emphasize that, although the $RL$ fractional derivatives are defined point-wise, due to the lack of information on the differential properties of the fractional integral 
and the singularities at the boundary in definition (see \cite{MR3684877,zhang2015total}), the space $BV^r$ introduced in \eqref{BV_too_large} is sometimes too large to work with. 
Therefore, to develop a satisfactory theory of fractional total variation we restrict the discussion to the smaller space $SV^r$ (see Definition \ref{strict_conv_BVr} below), 
which enjoys several advantages, such as approximation via smooth functions. We shall later show in Section \ref{image_application_sec} that, under certain restrictions (compatible in imaging processing problem) on the boundary values, 
 the restricted space $SV^r$ is equivalent to the full space $BV^r$.\\\\
In this article we focus mainly on the theoretical analysis of $TV^{r}$, with varying orders $r\in \R^+$. The existence of optimal solutions of the scheme $\mathcal T$, 
as well as the numerical realization and demonstration, will be the main topic of our second work \cite{panxinyang2018realordertraining} of this series. \\\\
The aim of this article is threefold. In the first part we study the basic properties of functions with bounded fractional order
total variation, such as the lower semi-continuity (\emph{l.s.c.}) with respect to the function, and compact embedding from $BV^r$ to $L^1$. In particular, we prove the following result:
\begin{theorem}[see Theorem \ref{ATV_real_embedding}]\label{real_embedding_intro}
Let $s\in (0,1)$ be given. Assume the sequence $\seqn{u_n}\subset SV^s(Q)\cap BV(Q)$ satisfies
\be
\sup\flp{\norm{u_n}_{L^\infty(\partial Q)}+\norm{u_n}_{BV^s(Q)}:\,\, n\in\N}<+\infty.
\ee
Then, there exists $u\in BV^s(Q)$ such that, upon subsequence, $u_n\to u$ strongly in $L^1(Q)$.
\end{theorem}
In the second part we investigate the functional properties of $\seqn{TV^{r_n}}$ with respect to a sequence of orders $\seqn{r_n}\subset \R^+$. In particular, the main theorem is:
\begin{theorem}[see Theorem \ref{compact_lsc_r_tv}]\label{compact_lsc_r_intro}
Given sequences $\seqn{r_n}\subset \R^+$ and $\seqn{u_n}\subset L^1(Q)$ such that $r_n\to r\in\R^+\cup \flp{0}$, and assume there exists $p\in(1,+\infty]$ such that 
\be
\sup\flp{\norm{u_n}_{L^p(Q)}+TV^{r_n}(u_n):\,\,n\in\N}<+\infty.
\ee
Then, the following statements hold.
\begin{enumerate}[1.]
\item
There exists $u\in BV^r(Q)$, such that, upon subsequence, $u_n\wto u$ weakly in $L^p(Q)$ and 
\be
\liminf_{n\to\infty} TV^{r_n}(u_n)\geq TV^r(u).
\ee
\item
Assuming in addition that ${u_n}\in SV^{r_n}(Q)\cap BV(Q)$, $\norm{u_n}_{L^\infty(\partial Q)}$ is uniformly bounded, and $r_n\to r>0$, we have that
\be
u_n\to u\text{ strongly in }L^1(Q).
\ee
\end{enumerate}
\end{theorem}
Lastly, in Section \ref{image_application_sec} we study functions $u\in BV^r(Q)$ vanishing on the boundary. 
In particular, Theorem \ref{smooth_approx_bdy_condition} shows that for a function $u\in BV^r(Q)$ such that all the terms
$\nabla^k u$ ($0\le k\leq \ir$), vanish on the boundary, there exists a sequence $\seqn{u_n}\subset C^\infty(Q)\cap BV^r(Q)$ such that 
\be
u_n\to u\text{ strongly in }L^1(Q)\text{ and }TV^r(u_n)\to TV^r(u).
\ee
The paper is organized as follows. In Section \ref{tvs_setting_prop} we collect some notations and preliminary results on 
the fractional order derivative. In Section \ref{main_body_section} we analyze the main properties of the 
fractional $r$-order total variation, with $r\in\R^+\setminus \N$ fixed. The compact embedding, the lower semi-continuity with respect to the order $r$, 
and relation between the fractional order total variation and its integer order counterpart will be 
the subjects of Section \ref{sec_functional_lsc}. 
Finally, in Section \ref{image_application_sec}, we show that the space $SV^r(Q)$ and $BV^r(Q)$ are equivalent under certain boundary conditions.
\section{Preliminary results on fractional order derivatives}\label{tvs_setting_prop}
Through this article, $r\in \R^+$ will denote a positive constant, 
and we will always write $r=\ir+s$ where $\ir$ denotes the integer part of $r$, and $s\in[0,1)$. \\\\
We collect the definitions of fractional order derivative in dimension one as follows.
\begin{define}[the fractional order derivative on unit interval]\label{frac_der_def}
Let $I:=(0,1)$ and $x\in I$ be given.
\begin{enumerate}[1.]
\item
The \emph{left (right)-sided Riemann-Liouville} derivative of order $r=\ir+s\in\R^+$ (see  \cite{MR1347689}) is defined by (resp.)
\be\label{R_L_frac_1d_left}
d^{r}_{L}w(x) = \frac{1}{\Gamma(1-s)}\fsp{\frac d{dx}}^{\ir+1}\int_0^x \frac{w(t)}{(x-t)^{s}}dt,
\ee 
and 
\be
d^{r}_{R}w(x) = \frac{(-1)^{\ir+1}}{\Gamma(1-s)}\fsp{\frac d{dx}}^{\ir+1}\int_x^1 \frac{w(t)}{(t-x)^{s}}dt,
\ee
where $\Gamma(\cdot)$ denotes the $\Gamma$-function, i.e. 
\[\Gamma(s):=\int_0^\infty e^{-t}t^{s-1}dt.\] 
\item
The \emph{left (right)-Riemann-Liouville} fractional order integrals (of order $r\in\R^+$) are defined by (resp.)
\be\label{r_int_frac_def}
(\mathbb I_{L}^r w)(x):=\frac{1}{\Gamma(r)}\int_0^x \frac{w(t)}{(x-t)^{1-r}}dt\,\text{ and }
\,({\mathbb I_R^r} w)(x):=\frac{1}{\Gamma(r)}\int_x^1 \frac{w(t)}{(x-t)^{1-r}}dt,
\ee 
\item
The {\em left (right)-sided \emph{Caputo} derivative} (of order $r$)
is defined by (resp.)
\be
d^r_{L,c} w(x):=\frac{1}{\Gamma(1-s)}\int_0^x \frac{(d^{\ir+1})w(t)}{(x-t)^{s}}dt.
\ee
and
\[
d^{r}_{R,c}w(x):= \frac{(-1)^{\ir+1}}{\Gamma(1-s)}\int_x^1 \frac{(d^{\ir+1})w(t)}{(t-x)^{s}}dt\]
\end{enumerate}
\end{define}
We collect some immediate results regarding Definition \ref{frac_der_def} from literatures
\begin{remark} 
Let $w\in C^\infty(\bar I)$ and $\phi\in C_c^\infty(I)$ be given.
\begin{enumerate}[1.]
\item
The following integration by parts formula hold (see, e.g.,\cite{1751-8121-40-24-003}), i.e.,
\be
\int_I w\,d_{R,c}^r \phi\,dx=(-1)^{\ir+1}\int_I (d_L^rw)\, \phi\,dx
\ee
and
\be
\int_I w\,d_{L,c}^r \phi\,dx=(-1)^{\ir+1}\int_I (d_R^rw)\, \vp\,dx
\ee
\item
The RL derivative and Caputo derivative are equivalent on compact supported functions (\cite[Theorem 2.2]{MR1347689}), i.e., for every $x\in I$, there holds
\be\label{caputo_eq_RL}
d^{r}_L\phi(x)= d^{r}_{L,c}\phi(x)\text{ and }d^{r}_R\phi(x)= d^{r}_{R,c}\phi(x).
\ee
Thus, we could re-write Assertions 1 to be
\be
\int_I w\,d_{R}^r \phi\,dx=(-1)^{\ir+1}\int_I (d_L^rw)\, \phi\,dx
\ee
and
\be
\int_I w\,d_{L}^r \phi\,dx=(-1)^{\ir+1}\int_I (d_R^rw)\, \vp\,dx
\ee
\item
For any functions $w_1$, $w_2\in C^\infty(\bar I)$, and $a$, $b\in\R$. the linearity holds, i.e.,
\be\label{linearity_frac}
d^r(a w_1(x)+b w_2(x))=a d^r w_1(x)+b d^r w_2(x)
\ee
\end{enumerate}
\end{remark}
\begin{remark}\label{left_use_whole}
In what follows we shall only work with Left-sided fractional order derivative/integrations, as the argument for Right-sided holds analogously. Thus, to simplify or notations, we shall drop the underlying $L$ and $R$ and only write $d^r$, $\mathbb I^r$, instead of $d^r_L$ and $\mathbb I^r_L$ unless specific otherwise.
\end{remark}
We next recall the definition of representable functions.
\begin{define}[Representable functions]\label{frac_represent_I}
We denote by $\mathbb I^r(L^1(I))$, $r>0$, the space of functions $f$ represented by the $r$-order derivative of a summable function. That is,
\be
\mathbb I^r(L^1(I)):=\flp{f\in L^1(I):\,\, f=\mathbb I^rw,\,\,w\in L^1(I)}.
\ee
\end{define}
Next we recall several theorems on representable functions in one dimension from \cite{MR1347689}.
\begin{theorem}\label{thm_MR1347689}
For convenience, we unify our notation by writing $\mathbb I^r = d^{-r}$ for $r<0$.
\begin{enumerate}[1.]
\item\label{cite_represent_frac}
{\cite[Theorem 2.3]{MR1347689}}
Condition $w(x)\in \mathbb I^r(L^1(I))$, $r>0$ is equivalent to
\be
(\mathbb I^{\ir+1-r}w)(x)\in W^{\ir+1,1}(I),\qquad r=\ir+s \label{inte_frac_cond}
\ee
and
\be
(d^l (\mathbb I^{\ir+1-r}w))(0)=0, \qquad l=0,1,\cdots,\ir\label{bdy_frac_cond}
\ee
\item\label{MR1347689T2_5}
{\cite[Theorem 2.5]{MR1347689}}
Let $w\in L^1(I)$ be given. The relation
\be
\mathbb I^{r_1} \mathbb I^{r_2} w=\mathbb I^{{r_1}+r_2}w
\ee
is valid if one of the following conditions holds:
\begin{enumerate}[1.]
\item
$r_2>0$, ${r_1}+r_2>0$, provided that $w\in L^1(I)$,
\item
$r_2<0$, ${r_1}>0$, provided that $w\in \mathbb I^{-r_2}(L^1(I))$,
\item
${r_1}<0$, ${r_1}+r_2<0$, provided that $w\in \mathbb I^{-{r_1}-r_2}(L^1(I))$.
\end{enumerate}
\item\label{semigroup_frac_int}
{\cite[Theorem 2.6]{MR1347689}}
Let $r\in\R^+$ be given.
\begin{enumerate}[1.]
\item
The fractional order integration operator $\mathbb I^r$ forms a semigroup in $L^p(I)$, $p\geq 1$, which is continuous in the uniform topology for all $r> 0$,
and strongly continuous for all $r\geq 0$. 
\item
It holds (see \cite[(2.72)]{MR1347689}) 
\be\label{eq_semigroup_frac_int}
\norm{\mathbb I^r w}_{L^1(a,b)}\leq (b-a)^r\frac{1}{r\Gamma(r)}\norm{w}_{L^1(a,b)}.
\ee
\end{enumerate}
\end{enumerate}
\end{theorem}
We close this section by introducing the notations for (partial) fractional order derivative in multi-dimensions. %
\begin{define}[fractional order partial derivative]
 Given $x=(x_1,\ldots,x_N)\in Q=(0,1)^N\subset\rn$ and $u\in C^\infty(Q)$, we define the $r$-order partial derivative: 
\begin{align*}
\partial^r_1 u(x) := \frac {d^r}{dt} u(t,x_2,x_3,\ldots,x_N),
\end{align*}
and similarly for $\partial_i^ru(x)$, $i=2,\ldots, N$.
\end{define}
%
%\begin{define}[multiindex fractional order derivative]\label{multiindex_fractional}
%Let $r\in\R^+$ and $\alpha_r=(\alpha_1,\ldots,\alpha_{\ir+1})$ is a multiindex of order $\abs{\alpha_r}=\alpha_1+\cdots+\alpha_{\ir+1}=\ir+1$. 
%Then we define the left fractional $\alpha_r$-order multi-index derivative as follows: for the case $\alpha_{\ir+1}=1$,
%\be
%\partial^{\alpha_r}_{L}u(x_1,x_2) =\partial^{\alpha_r}\fmp{ \frac{1}{\Gamma(1-s)}\int_0^{x_1} \frac{u(t,x_2)}{(x_1-t)^{s}}dt}.
%\ee
%The case $\alpha_{\ir+1}=2$ is defined similarly. Also, we define $\partial^r_Ru(x_1,x_2)$, and $\partial^r_Cu(x_1,x_2)$ in a similar way.
%\end{define}
%
%
%
%
We next recall the integration by parts formula of fractional order from \cite{chen2013fractional}: for $u\in C^\infty(Q)$, $v\in C_c^\infty(Q)$, it holds
\be
\int_Q u\,\partial^s v \,dx = - \int_Q \partial^s u\, v\,dx.
\ee
Then, by Theorem \ref{thm_MR1347689}, Assertion \ref{MR1347689T2_5}, we have the following multi-index integration by parts formula:
\be
\int_Q u\,\partial^r v \,dx = (-1)^{\ir+1} \int_Q \partial^r u\, v\,dx.
\ee
We conclude this section by recalling the following technical lemma. Let $k\in\N$ be given.
\begin{lemma}[\cite{MR1347689}]\label{power_function_s}
For every $s\in[0,1)$ and $x\in I$ the following assertions hold.
\begin{enumerate}[1.]
\item
We have
\be
d^s_L x^k = \frac{\Gamma(k+1)}{\Gamma\fsp{k-s+1}}x^{k-s}.
\ee
\item
If $k=0$, i.e., $x^k=1$ a constant, then $d^s_L x^0=0$ if and only if $s\in\N$.
\item
For all $s\in(0,1)$, we have $d^s_L x^{s-1}=0$.
\end{enumerate}
\end{lemma}
%\begin{proof}
%We first prove statement 1. We observe that, from \eqref{R_L_frac_1d_left}, 
%\begin{align}
%d^s_L x^k &=  \frac{1}{\Gamma(1-s)}{\frac d{dx}}\int_0^x \frac{t^k}{(x-t)^{s}}dt\notag\\
%&= \frac{1}{\Gamma(1-s)}{\frac d{dx}}\fmp{x^{1-s+k}\int_0^1 \frac{y^k}{(1-y)^{s}}dy}\\
%&=\frac{1}{\Gamma(1-s)}{\frac d{dx}}\fmp{x^{1-s+k}B(1-s,k+1) }.\label{power_drop_1}
%\end{align}
%Here $B(1-s,k+1)$ denotes the \emph{Beta Euler function}, and we observe that
%\be
%B(1-s,k+1) = \frac{\Gamma(1-s)\Gamma(k+1)}{\Gamma(2-s+k)}.
%\ee
%This, together with \eqref{power_drop_1}, we have 
%\be
%d^s_L x^k = \frac{1}{\Gamma(1-s)}{\frac d{dx}}\fmp{x^{1-s+k}\frac{\Gamma(1-s)\Gamma(k+1)}{\Gamma(2-s+k)} } = \frac{\Gamma(k+1)}{\Gamma(k-s+1)}x^{k-s},
%\ee
%and we infer Statement 1. Statement 2 follows from the computations above. For Statement 3, by \eqref{power_drop_1} we have
%\be
%d^s_L x^k =\frac{1}{\Gamma(1-s)}{\frac d{dx}}\fmp{x^{1-s+s-1}B(1-s,\ir+1) }=0,
%\ee
%concluding the proof.
%\end{proof}
%
%
\section{The space of functions with bounded fractional-order total variation}\label{main_body_section}
\subsection{Total variation with different underlying Euclidean norm}
We start by recalling the definition of Euclidean $\ell^p$-norm on $\rn$. Let $p\in[1,+\infty)$ and $x=(x_1,x_2,\ldots,x_N)\in \rn $ be given, we define 
\be
\abs{x}_{\ell^p}:=(|x_1|^p+|x_2|^p+\cdots |x_N|^p)^{1/p},
\ee
and we note that $\abs{\cdot}_{\ell^p}$ are equivalent norms on $\rn$. That is, for any $1\leq q<p\leq \infty$, we have
\be\label{x_eu_equivalence}
\abs{x}_{\ell^p}\leq\abs{x}_{\ell^p}\leq N^{1/q-1/p}\abs{x}_{\ell^p}.
\ee
\begin{define}\label{TV_ATV_integer}
Let $u\in L^1(Q)$ be given. We recall the following definition.
\begin{enumerate}[1.]
\item
The first order total variation with underlying Euclidean $\ell^p$-norm:
\be
TV_{\ell^p}(u):=\sup\flp{\int_Qu\, \divg \vp \,dx:\,\,\vp\in C_c^\infty(Q;\rn)\text{ and }\abs{\vp}_{\ell^p}^\ast\leq 1},
\ee
where $\abs{\cdot}_{\ell^p}^\ast$ denotes the dual norm associated with $\abs{\cdot}_{\ell^p}$. 
\item
The second order total variation with underlying Euclidean $\ell^p$-norm:
\be
TV_{\ell^p}^2(u):=\sup\flp{\int_Qu\, \divg^2 \vp \,dx:\,\,\vp\in C_c^\infty(Q;\M^{N\times N})\text{ and }\abs{\vp}_{\ell^p}^\ast\leq 1}.
\ee
\item
Generally, the $k$-th order total variation, $k\in\N$, with underlying Euclidean $\ell^p$-norm:
\be
TV_{\ell^p}^k(u):=\sup\flp{\int_Qu\, \divg^k \vp \,dx:\,\,\vp\in C_c^\infty(Q;\M^{N\times (N^{k-1})})\text{ and }\abs{\vp}_{\ell^p}^\ast\leq 1}.
\ee
\end{enumerate}
\end{define}
For example, when $N=2$, we have $\vp=[\vp_1,\vp_2;\vp_3,\vp_4]$ and 
\begin{align*}
\divg^2 \vp = \divg(\divg(\vp_1,\vp_2),\divg(\vp_3,\vp_4))&=\divg(\partial_1 \vp_1+\partial_2 \vp_2, \partial_1\vp_3+\partial_2\vp_4)\\
&=\partial_1\partial_1\vp_1+\partial_1\partial_2\vp_2+\partial_2\partial_1 \vp_3+\partial_2\partial_2\vp_4.
\end{align*}
\textbf{Remark.} We could recovery the classical \emph{isotropic total variation} ($TV$) and \emph{an-isotropic total variation} ($ATV$) by letting $p=2$ and $p=1$, respectively. Moreover, we note that in dimension one, all $TV_{\ell^p}$ are the same.\\\\
We recall the usual trace operator for function with bounded total variation.
\begin{theorem}[{\cite[Theorem 2, Page 181]{evans2015measure}}]\label{usual_trace}
Suppose $u\in BV(Q)$. Then for $\hnmo$ a.e. $x_0\in\partial Q$,
\be
\lime\fint_{B(x_0,\e)\cap Q}\abs{u- T[u](x)}dx=0,
\ee
where $T[\cdot]$ denotes the standard trace operator. In another word, we have
\be
T[u](x_0)=\lime \fint_{B(x_0,\e)\cap Q}u(x)\,dx.
\ee
\end{theorem}

\begin{define}\label{isotropic_frac_variation}
We define the \emph{$r$-order total variation} $TV_{\ell^p}^r(u)$ on $u\in L^1(Q)$ as follows.
\begin{enumerate}[1.]
\item
For $r=s\in(0,1)$ (i.e. $\ir=0$), we define
\be\label{RLFOD}
TV_{\ell^p}^s(u):=\sup\flp{\int_Qu\, \divg^s \vp \,dx:\,\,\vp\in C_c^\infty(Q;\rn)\text{ and }\abs{\vp}_{\ell^p}^\ast\leq 1},
\ee  
where we set
\be\label{scaled_divg}
\divg^s u:=[(1-1/N)s+1/N]\sum_{i=1}^N\partial^s_{i,R}\vp_i; 
\ee
\item
For $r=\ir+s$ where $\ir\geq1$, we define 
\be\label{RLFODr}
TV_{\ell^p}^r(u):=\sup\flp{\int_Qu\, \divg^s[\divg^\ir \vp] \,dx:\,\,\vp\in C_c^\infty(Q;\M^{N\times (N^{k})})\text{ and }\abs{\vp}_{\ell^p}^\ast\leq 1}.
\ee
\end{enumerate}
\end{define}
\begin{remark}
In \eqref{scaled_divg} we applied the Right-sided derivative on the test function $\vp$. We do this so that when $u$ is sufficient regular, for example $u\in C^\infty(\bar Q)$, the integration by parts formula holds so that 
\be
\int_Qu\, \divg^s \vp \,dx =- \int_Q\nabla^s_Lu\,\vp\,dx.
\ee 
That is, we have Left-sided operator on function $u$, as we mentioned in Remark \ref{left_use_whole} that we primally work on Left-sided operator in this article. Moreover, if we choose to work primally on Right-sided derivative, we shall use Left-sided derivative on test function $\vp$ in \eqref{scaled_divg}.
\end{remark}
%
%
%\Pblue
\begin{define}\label{strict_conv_BVr}
Let $r\in\R^+$ and $p\in[1,+\infty]$ be given. 
\begin{enumerate}[1.]
\item
Given a sequence $\seqn{u_n}\subseteq L^1(Q)$ such that $TV^r(u_n)<+\infty$ for each $n\in\N$, we say it is \emph{strictly} converging to $u\in L^1(Q)$ 
with respect to the $TV^r_{\ell^p}$ seminorm, and write $u_n\wtogsp u$, if
\be\label{def_eq_strict_covg}
\limn\norm{u-u_n}_{L^1(Q)}+\abs{TV^r_{\ell^p}(u_n)-TV_{\ell^p}^r(u)}= 0.
\ee
That is, $u_n\wtogsp u$ if $u_n\to u$ strongly in $L^1(Q)$ and $TV^r_{\ell^p}(u_n)\to TV_{\ell^p}^r(u)$.

\item
We define the space $SV^r(Q)$ by
\be SV^r(Q):=\bigcap_{p\in[1,+\infty]} \overline{C^\infty(Q)}^{\text{ ss-}(p)} .
\label{BV_r_p_diff_ss} \ee
That is, $SV^r(Q)$ is the intersection (among all $p\in[1,+\infty] $) of the closures of $C^\infty(Q)$ with respect to the ss-$(p)$ convergence.
% \be
% SV^r(Q):=\bigcap_{p\in[1,+\infty]}\operatorname{cl}( C^\infty(Q)\text{ with respect to strict }\text{ss-(p)}\,\,TV^r_{\ell^p}\text{ norm})
% \ee
\item
We define the (standard) $BV^r(Q)$ space by 
\be\label{BV_r_p_diff}
BV^r(Q):=\bigcap_{p\in[1,+\infty]}\flp{u\in L^1(Q):\,\, TV^r_{\ell^p}(u)<+\infty}.
\ee
\end{enumerate}
\end{define}
\begin{remark}[Equivalence between $TV_{\ell^p}^r$]\label{an_iso_equ}
Definition \ref{strict_conv_BVr} has several consequences.
\begin{enumerate}[1.]
\item
By \eqref{x_eu_equivalence} we have that, for any $1\leq q<p\leq+\infty$,
\be\label{equ_p_1_p}
N^{1/p-1/q}TV_{\ell^p}^r(u)\leq TV_{\ell^q}^r(u)\leq TV_{\ell^p}^r(u).
\ee
That is, the set $\flp{u\in L^1(Q):\,\,TV^r_{\ell^p}(u)<+\infty}$ is in reality independent of $p\in[1,+\infty]$. Thus, we could simplify the real $r$-order bounded variation space $BV^r(Q)$ defined in \eqref{BV_r_p_diff} to
\be
BV^r(Q):=\flp{u\in L^1(Q):\,\,TV^r_{\ell^2}(u)<+\infty}.
\ee
without dependence on the underlying $\ell^p$-norm.
\item %\RRR 
The space $SV^r(Q)$, defined in \eqref{BV_r_p_diff_ss}, enjoys the ``smooth approximation" property: for each $u\in SV^r(Q)$ and $p\in[1,+\infty]$, there exists 
a sequence $\seqn{u_n}\subseteq C^\infty(Q)\cap BV^r(Q)$ such that \eqref{def_eq_strict_covg} holds. 
In the case of integer order, i.e. $r=k\in\N$, we do have $SV^k(Q)=BV^k(Q)$ (see for instance \cite{evans2015measure}). 
However, due to the singularities at the boundary arising from the definition of fractional derivatives, we are unable to prove a smooth approximation result.
In particular, the construction from \cite{evans2015measure} would not work, unless additional conditions are assumed. In Section \ref{sec_functional_lsc} 
and Theorem \ref{smooth_approx_bdy_condition}, we will discuss certain special conditions, compatible with the imaging processing problems, where such smooth approximation 
results could be proven. 
%(XYL: this part seems ok).
\end{enumerate}
\end{remark}

\BLK
We conclude this subsection with some definitions and properties of fractional order Sobolev seminorms.
\begin{define}
For $r\in\R^+$, we define the $r$-order fractional Sobolev space by 
\be
W^{r,1}(Q):=\flp{u\in L^1(Q;\mathbb M^{N\times N^{\ir}}):\,\, \exists g\in L^1(Q)\text{ such that }\int_Q u\, \divg_R^r \vp \,dx=\int_Q g\,\vp\,dx},
\ee
where $g\in L^1(Q;\mathbb M^{N\times N^{\ir}})$ is the weak $r$-order fractional derivative of $u$. Also, we equip it with norm
\be
\norm{u}_{W^{r,1}_{\ell^p}(Q)}:=\norm{u}_{L^1(Q)}+\int_Q\abs{g}_{\ell^p}dx.
\ee
\end{define}

\begin{define}
Let $r=\ir+s$. By $AC^{r,1}(I)$ we denote the set of all functions $w:I\to\R$ admitting a representation
of the form
\be\label{repre_AC}
w(t)=\sum_{i=0}^\ir \frac{c_i}{\Gamma(s+i)}t^{s-1+i}+\mathbb I^r \phi(t),\,\, t\in I\,\,a.e.,
\ee
where $c_0,\ldots,c_k\in\R$ and $\phi\in L^1(I)$.
\end{define}
We recall the following results form \cite{MR3144452}.
\begin{theorem}\label{thm_MR3144452}
Let $r=\ir+s$ be given.
\begin{enumerate}[1.]
\item
\cite[Theorem 7]{MR3144452} the function $w\in L^1(I)$ admits the $r$-order derivative if and only if $w\in AC^{r,1}(I)$. In this case, $w$ has the representation \eqref{repre_AC} and we have
\begin{align}\label{d_representable}
d^{i+s} w(0)=c_i,\,\, i=0,\ldots, k-2,\text{ and }d^r w(t)=\phi(t),\,\, t\in I\,\,a.e..
\end{align}
\item
\cite[Theorem 19]{MR3144452} we have
\be\label{ac_sobolev_s}
W^{r,1}(I)=AC^{r,1}(I)\cap L^1(I).
\ee
\end{enumerate}
\end{theorem}
\textbf{Remark.} From \eqref{d_representable} and \eqref{ac_sobolev_s}, we see that for $w\in W^{r,1}(I)$, the corresponding $\phi\in L^1(I)$ satisfies $\norm{\mathbb I^r\phi}_{L^1(I)}<+\infty$, 
and hence $\phi\in \mathbb I^r(L^1(I))$ in view of Definition \ref{frac_represent_I}.
\begin{remark}\label{int_iso_equ}
Let $p\in[1,+\infty]$ be given, and assume that $u\in C^\infty(Q)\cap BV^{s}(Q)$. Then we have
\be\label{TV_sobolevl_equiv}
TV^s_{\ell^p}(u) =\int_Q \abs{\nabla^s u}_{\ell^p}dx=\int_Q \abs{(\partial_1^s u,\partial_2^su)}_{\ell^p}dx.
\ee
The proof follows directly from \cite[Proposition 3.5]{zhang2015total}.\\\\

Equation \eqref{TV_sobolevl_equiv} allows us to use the fact that the an-isotropic total variation $TV_{\ell^1}$ can be computed axis by axis. For example, let $N=2$, by \eqref{TV_sobolevl_equiv} we have
\be
 TV^s_{\ell^1}(u) =\int_Q \abs{(\partial_1^s u,\partial_2^su)}_{\ell^1}dx= \int_Q \abs{\partial^s_1 u}dx+\int_Q \abs{\partial^s_2 u}dx.
\ee
That is, we are able to separate the integration in the two variables. This will be crucial in allowing us to study
properties of real order total variation in the multi-dimensional setting, by using results from the (often easier) one dimensional case. 

\end{remark}

\subsection{Basic properties of $TV^r$ with fixed order of derivative}\label{sec_smooth_approx}%\Pblue
For brevity, we will only write $TV^r$ without explicit reference to the underlying Euclidean $\ell^p$-norm. 
However, in several arguments it will be advantageous to use the $TV_{\ell^1}^r$ seminorm (see Remark \ref{int_iso_equ}). 
 In such instances, we will thus write $TV_{\ell^1}^r$ to clarify that we are relying on the underlying norm being the Euclidean $\ell^1$-norm.
We first prove a lower semi-continuity result with fixed order $r\in\R^+$.

\begin{theorem}  \label{weak_star_comp_s} %[lower semi-continuity with respect to fixed order $r\in\R^+$]
Given $u\in \mb(Q)$, and sequence $\seqn{u_n}\subseteq BV^r(Q)$ satisfying one of the following conditions:
\begin{enumerate}[1.]
\item
$u\in L^1(Q)$ and $\seqn{u_n}$ is locally uniformly integrable and $u_n\to u$ a.e.,
\item 
$u_n\wtos u$ in $\mb(Q)$
\item 
 $u_n\wto u$ in $L^p(Q)$, for some $p>1$,
\end{enumerate}
then we have
\be\label{liminf_s_weak}
\liminf_{n\to\infty} TV^r(u_n)\geq TV^r(u).
\ee
\end{theorem}
Note that (3) is stronger than (2), but we stated it explicitly since it constitutes a special case widely used in this paper.
To prove Theorem \ref{weak_star_comp_s}, a preliminary result is required. 

\begin{lemma}\label{linftyds}
Let $\vp\in C_c^\infty(Q)$ be given. Then the following statements hold.
\begin{enumerate}[1.]
\item 
For any fixed $T\in \N$, we have
\be
\sup\flp{\norm{d^r\vp}_{L^\infty(Q)}:\,\,{r\in(0,T)} }<+\infty.
\ee
\item
For a.e. $x\in Q$, we have
\be
\divg^r\vp(x)\to \divg^\ir\vp(x)\text{ as }r\to\ir^+
\ee
and 
\be
\divg^r\vp(x)\to \divg^{\lceil r \rceil}\vp(x)\text{ as }r\to \lceil r \rceil^-.
\ee
\end{enumerate}
\end{lemma}

\begin{proof}
We first consider the one dimensional case, i.e. $Q=I=(0,1)$. %Let $\vp\in C_c^\infty(I)$ be given. 
In view of \eqref{caputo_eq_RL}, for any $r=\ir+s$, we have
\begin{align*}
\abs{d^r_Lw(x)}& = \abs{d^r_{L,c}(x)}\leq\frac{1}{\Gamma(1-s)}\int_0^x \frac{\abs{(d^{\ir+1})w(t)}}{(x-t)^{s}}dt\\
&\leq \norm{w}_{W^{\ir+1,+\infty}(I) }\cdot\frac{1}{\Gamma(1-s)}\int_0^x \frac{1}{(x-t)^{s}}dt\\
&\leq \norm{w}_{W^{\ir+1,+\infty}(I)}\cdot\frac{1}{(s-1)\Gamma(1-s)}\fmp{(x-t)^{1-s}\bigg|^x_0}\\
&\leq \norm{w}_{W^{\ir+1,+\infty}(I)}\cdot\frac{1}{(1-s)\Gamma(1-s)},
\end{align*}
which implies 
\be
\abs{d^r_L w(x)} \leq\norm{d^{\ir+1}w}_{L^{\infty}(I)}\cdot\frac{1}{(1-s)\Gamma(1-s)}.
\ee
Note also that, by Euler's reflection formula,
\be
\sup_{s\in(0, 1)}\frac{1}{\Gamma(1-s)(1-s)}=\sup_{s\in(0,1)}\frac{1}{\Gamma(2-s)}\leq1,
\ee
hence, 
\be
\norm{d_L^rw}_{L^\infty(I)}\leq \norm{d^{\ir+1}w}_{L^{\infty}(I)} <+\infty,
\ee
and
\be
\sup\flp{\norm{d^r_L\vp}_{L^\infty(I)}:\,\,{r\in(0,T)} }\leq \sum_{l=0}^{T}\norm{d^{l+1}\vp}_{L^{\infty}(I)}<+\infty.
\ee
The same arguments give the desired results for the right and central sided $RL$ derivatives.\\\\
We next prove Statement 2, for the case $0<r<1$. 
The case $r\geq 1$ is proven using similar arguments. In view of \eqref{caputo_eq_RL}, we have $d^s\vp=d^s_c \vp$, with
$d^s_c $ denoting the \emph{Caputo} fractional derivative. Recall the Laplace transform gives
\be
\mathcal L\flp{d^s_c \vp}(y)=y^s\mathcal L\flp{\vp}(y)-y^{s-1}\vp(0)=y^s\mathcal L\flp{\vp}(y).
\ee
Therefore, we have 
\be
\lim_{s\nearrow 1}\mathcal L\flp{d^s_c \vp}(y) = y\,\mathcal L\flp{\vp}(y) 
\qquad\text{ and }\qquad\lim_{s\searrow 0}\mathcal L\flp{d^s_c \vp}(y) = \mathcal L\flp{\vp}(y),
\ee
and hence we conclude that 
\be
\lim_{s\nearrow 1}d^s_c(\vp)(x)= d\vp(x)\qquad\text{ and }\qquad\lim_{s\searrow 0}d^s_c(\vp)(x)= \vp(x),
\ee
as desired. \\\\
The multidimensional case can be directly inferred from one dimensional case.
We discuss the two dimensional case (i.e. $N=2$) as example. Recall that we defined 
\[\divg^s \vp(x_1,x_2)=\partial^s_1\vp(x_1,x_2)+\partial_2^s\vp(x_1,x_2),\] 
thus we have $\partial^s_1\vp(x_1,x_2) = d^s\vp(x_1,x_2)\lfloor_{x_2}$ for any fixed $x_2\in (0,1)$.
\end{proof}

\begin{proof}[Proof of Theorem \ref{weak_star_comp_s}]
Let $r=\ir+s$ be given, and fix an arbitrary 
$\vp\in C_c^\infty(Q;\M^{N\times (N^{k})})$. In view of \eqref{RLFOD} (or \eqref{RLFODr}) we have
\be
TV^r(u_n)\geq \int_Qu_n\, [\divg^s\divg^\ir \vp] \,dx.
\ee
If we are assume condition (1) holds. Since $\seqn{u_n}$ is locally uniformly integrable and $u_n\to u$ a.e., 
by the dominated convergence theorem 
\be
\limn\, \int_Qu_n\, [\divg^s\divg^\ir \vp] \,dx =\int_Qu\, [\divg^s\divg^\ir \vp] \,dx.
\ee
If we are assume either condition (2), or the even stronger (3), note that
 since $\vp\in C_c^\infty(Q;\mathbb M^{N\times (N)^\ir})$, we have $\divg^s\divg^\ir\vp\in C(\bar Q)$. Hence, since $\seqn{u_n}\subset L^1(Q)$ and in view of the $\text{weak}^\ast$ convergence in $\mb(Q)$ (see \cite[Page 116]{brezis2010functional}), we again have
\be
\limn\, \int_Qu_n\, [\divg^s\divg^\ir \vp] \,dx =\int_Qu\, [\divg^s\divg^\ir \vp] \,dx.
\ee
%Moreover, we obtain same conclusion if condition (3) holds.\\\\
%
Thus, in all cases, we have
\be
\liminfn\, TV_{\ell^1}^r(u_n)\geq \liminfn\, \int_Qu_n\, [\divg^s\divg^\ir \vp] \,dx =\int_Qu\, [\divg^s\divg^\ir \vp] \,dx.
\ee
Taking the supremum over all $\vp\in C_c^\infty(Q;\M^{N\times(N^\ir)})$ with $\abs{\vp}\leq 1$, we conclude \eqref{liminf_s_weak}, as desired. 
\end{proof}

We recall the \emph{Riesz} representation theorem.
\begin{theorem}[{\cite[Theorem 1, Section 1.8]{evans2015measure}}]\label{riesz_theorem}
Let $L:$ $C_c(\rn,\R^M)\to\R$ be a linear functional satisfying
\be
\sup\flp{L(\vp):\,\,\vp\in C_c(\rn;\R^M),\,\,\abs{\vp}\leq 1,\,\,\operatorname{spt}(\vp)\subset K}<+\infty
\ee
for each compact set $K\subset \rn$. Then there exists a Radon measure $\mu$ on $\rn$, and a $\mu$-measurable function $\sigma$: $\rn\to\R^M$ such that
\begin{enumerate}[1.]
\item
$\abs{\sigma(x)}=1$ for $\mu$-a.e. $x$, and
\item
$L(\vp)=\int_\rn \vp\cdot\sigma d\mu$.
\end{enumerate}
\end{theorem}

Next, we claim that the $TV^r$ semi-norm is a \emph{Radon} measure.
\begin{lemma}\label{jinleyidian}
Given a functions $u\in BV^r(Q)$, there exists a Radon measure 
$\mu$ on $Q$ and a $\mu$-measurable function $\sigma: Q\to\rn$ such that 
\begin{enumerate}[1.]
\item
$\abs{\sigma(x)}=1$ $\mu$-a.e., and
\item
$\int_Q u\,\divg^r \vp\,dx=-\int_Q \vp\cdot\sigma \,d\mu$ for all $\vp\in C_c^\infty(Q;\rn)$.
\end{enumerate}
\end{lemma}
\begin{proof}
We first define the linear functional 
\be
L:\,\, C_c^\infty(Q;\mathbb M^{N\times(N^\ir)})\to\R,
\qquad L(\vp):=-\int_Q u\,\divg^r\vp\,dx.
\ee
 Since $TV^r(u)<+\infty$, we have, in view of \eqref{RLFODr},
\be
\sup\flp{\frac1{\norm{\vp}_{L^\infty(Q)}}\int_Qu\, \divg^r \vp \,dx:\text{ for }\vp\in C_c^\infty(Q;\M^{N\times (N^\ir)})}=TV^r(u)<+\infty.
\ee
Thus, 
\be\label{smooth_define}
\abs{L(\vp)}\leq TV^r(u)\norm{\vp}_{L^\infty(Q)}.
\ee
Now we extend by continuity the definition of $L$ to the entire space $C_c(Q;\rn)$:
 for an arbitrary $\vp\in C_c(Q;\rn)$, we consider the mollifications $\vp_\e:=\vp\ast\eta_\e$ 
(for some ininfluent mollifier $\eta_e$) and, by \cite[Theorem 1, item (ii), Section 4.2]{evans2015measure}, 
\be\label{uniform_approach}
\vp_\e\to\vp\text{ uniformly on }Q.
\ee
Therefore, we by defining
\be
\bar L(\vp):=\lime L(\vp_\e)\text{ for }\vp\in C_c(Q;\rn),
\ee
in view of \eqref{smooth_define} and \eqref{uniform_approach}, we conclude that 
\be
\sup\flp{\bar L(\vp):\text{ for }\vp\in C_c(Q;\M^{N\times (N^{k})})\text{ and }\abs{\vp}\leq 1}<+\infty.
\ee
Thus, by Theorem \ref{riesz_theorem}, the proof is complete.
\end{proof}
%
%\RRR
Next, by relying on the geometry of the domain $Q=(0,1)^N$, 
we can further restrict the smooth approximation function of $u\in SV^r(Q)$ by forcing them
to belong to $C^\infty(\bar Q)$.
\begin{proposition}[strict approximation with smooth functions]\label{approx_smooth}
Let $r\in\R^+$ and $u\in SV^r(Q)$ be given. Then there exists a sequence $\seqn{u_n}\subset C^\infty(\bar Q)\cap BV^r(Q)$ such that 
\be\label{strictly_ATV_approx}
u_n\to u\text{ strongly in }L^1(Q)\text{ and }\limn {TV^r}(u_n)={TV^r}(u).
\ee
\end{proposition}

\begin{proof}
Let $u\in SV^r(Q)$ be given. In view of Definition \ref{strict_conv_BVr}, there exists sequence $\seqn{u_n}\subset C^\infty(Q)$ such that 
\be\label{wtogs_un_u}
u_n\wtogs u.
\ee
Next, let $x_0:=(1/2,\cdots,1/2)$ be the center of $Q$. For sufficiently small $\e>0$, we define
\be\label{u_n_eps_construct}
u_n^\e(x):=u_n\Big(\frac{x-x_0}{1+\e}+x_0 \Big)\text{ for }x\in Q.
\ee
Note that by construction, $u_n^\e$ will be a scaled version of the restriction $u_{n}$ to $Q_\e:={Q}/\fsp{1+\e}$.
Since $u_n\in C^\infty(\overline{Q_\e})$, we have $u_n^\e\in C^\infty(\bar Q)$ too. \\\\
We next show that $TV^r(u^\e_n)\to TV^r(u_n)$. Let $\e>0$ be fixed, then clearly we have (see e.g. Lemma \ref{jinleyidian})
\be\label{smaller_set_tv}
TV^r(u_n)\geq \sup\flp{\int_{Q_\e} u_n\divg^r\vp\,dx:\,\,\vp\in C_c^\infty(Q_\e), |\vp|\le 1}.
\ee
 On the other hand, for any $\vp\in C_c^\infty (Q)$, we have
\begin{align*}
\int_Q u_n^\e\divg^r\vp\,dx 
&=\int_Q u_n\fsp{(x-x_0)/(1+\e)+x_0}\divg^r\vp(x)\,dx\\
 &\leq (1+\e)^{\ir+1} \int_{Q_\e} u_n(y)\,\divg^r\vp((y-x_0)(1+\e)+x_0) dy\\
 &\leq (1+\e)^{\ir+1} TV^r(u_n),
\end{align*}
where in the last inequality we used \eqref{smaller_set_tv}. Hence, we have 
\be\label{limsup_simple_geometry}
TV^r(u_n^\e)\leq (1+\e)^{\ir+1} TV^r(u_n).
\ee
On the other hand, in view of \eqref{u_n_eps_construct}, we have $u_n^\e\to u_n$ strongly in $L^1(Q)$. By Proposition \ref{weak_star_comp_s}, we obtain that 
\be
\liminf_{\e\to 0}TV^r(u_n^\e)\geq TV^r(u_n).
\ee
This, combined with \eqref{limsup_simple_geometry}, implies that $u_n^\e\wtogs u_n$. 
Combined with \eqref{wtogs_un_u}, we infer the existence of a subsequence 
$\seqn{u_{\e_n}}\subset C^\infty(Q)$ such that $u_{\e_n}\wtogs u$, which concludes
the proof of \eqref{strictly_ATV_approx}.
\end{proof}
%
%\begin{corollary}\label{coro_approx_smooth}
%Following the same setting in Theorem \ref{approx_smooth} and let $p\in[1,+\infty]$ be given. Then there exists a sequence $\seqn{u_n}\subset C^\infty(\bar Q)\cap BV^r(Q)$ such that 
%\be
%u_n\to u\text{ strongly in }L^1(Q)\text{ and }\limn {TV_{\ell^p}^r}(u_n)={TV_{\ell^p}^r}(u).
%\ee
%\end{corollary}
%\begin{proof}
%see appendix
%\end{proof}
\BLK
\begin{remark}\label{smooth_use_here}
We again emphasize that, in view of Remark \ref{int_iso_equ} and \cite[Proposition 3.5]{zhang2015total}, we have for $u\in C^\infty(Q)\cap BV^r(Q)$,
\be
TV_{\ell^1}^{l+s}(u)=\int_Q \abs{\nabla^{l+s}u}dx=\sum_{\abs{\alpha_s}=l+1}\int_0^1\cdots\int_0^1\abs{\partial^{\alpha_s} u(x_1,\ldots,x_N)}dx_1\ldots dx_N.
\ee
This allows us to compute multi-dimensional an-isotropic total variation $TV_{\ell^1}^{l+s}(u)$, by computing one dimension total variations $TV^{l'+s}(w)$, $l'=0,1,\ldots, l$, $w\in L^1(I)$, along each coordinate axis.
\end{remark}

\section{Analytic properties of function space $SV^r(Q)$}\label{sec_functional_lsc}

\begin{remark}
As the goal of this article is to construct models specifically for imaging applications, we assume the function $u\in L^1(Q)$ analyzed here represents an image.
That is, we could restrict ourself to consider only functions in the space
\be\label{image_function}
\IM(Q):=\flp{u\in BV(Q):\,\,\norm{T[u]}_{L^\infty(\partial Q)}\leq 1}.
\ee
We remark that most of the conclusions of this article are independent of \eqref{image_function}. However, certain results can be improved 
with \eqref{image_function}. We shall always highlight the assumption of \eqref{image_function} when we use it. We also note that assumption \eqref{image_function} is crucial in numerical realization of fractional order derivative. We refer readers to \cite[Section 4]{chen2013fractional} and the references therein.
\end{remark}
\subsection{Compact embedding for $BV^s$ with fixed $s\in(0,1)$}\label{sec_compact_fixed_s}
We start by recalling the following theorem from \cite{brezis2010functional}.
\begin{theorem}[{\cite[Theorem 4.26]{brezis2010functional}}]\label{lp_compactness}
Let $\mathcal F$ be a bounded set in $L^p(\rn)$ with $1\leq p<+\infty$. Assume that 
\be
\lim_{\abs{h}\to 0}\norm{\tau_h f-f}_{L^p(\rn)}=0\text{ uniformly in }f\in\mathcal F.
\ee
Then, the closure of $\mathcal F\lfloor \Om$ in $L^p(\Om)$ is compact for any measurable set $\Om\subset \rn$ with finite measure.
\end{theorem}
The main result of Section \ref{sec_compact_fixed_s} reads as follows.
\begin{theorem}[compact embedding $BV^s(Q) \hookrightarrow L^1(Q)$]\label{ATV_real_embedding}
Let $s\in (0,1)$ be given. Assume $\seqn{u_n}\subset \text{ss-}BV^s(Q)$ satisfy
\be\label{image_condition_infty}
\sup\flp{\norm{u_n}_{L^\infty(\partial Q)}+\norm{u_n}_{BV^s(Q)}:\,\, n\in\N}<+\infty.
\ee
Then, there exists $u\in BV^s(Q)$ such that, upon subsequence, $u_n\to u$ strongly in $L^1(Q)$.
\end{theorem}
We prove Theorem \ref{ATV_real_embedding} in several steps. We first recall a revisited left sided $RL$ $s$-order derivative
from \cite{MR2237634},
which we denote by $\hat d_L^s w(x)$, as follows:
\be\label{revisied_RL}
\hat d^s_Lw(x):=\frac1{\Gamma(1-s)} \frac d{dx}\int_0^x \frac{w(t)-w(0)}{\fsp{x-t}^s}dt.
\ee
%\Pblue
That is, the singularity of $d^s_L$ at the boundary $t=0$ is removed. 
Moreover, we remind that in dimension one, the semi-norms $TV$ and $TV_{\ell^1}$ are the equivalent.\BLK
\begin{proposition}\label{one_d_cuoweixiangjian}
Let $w\in BV^s(I)\cap C^\infty(I)$ be given. Let $TV^s(w,I)$ denote the $s$-order total variation of $w$ in $I=(0,1)$ and
\be
\tilde w(x):=
\begin{cases}
w(x)&\text{ for }x\in I\\
0&\text{ if }x\in \R\setminus I
\end{cases}
\ee
Then we have
\be\label{dengchaxiangjian_p1}
\norm{\tau_h\tilde w-\tilde w}_{L^1(\R)}\leq h^s \fsp{TV^s(w)+\norm{w}_{L^\infty(\partial I)}},
\ee
provided that $\norm{w}_{L^\infty(\partial I)}<+\infty$.
\end{proposition}

\begin{proof}
We start by recalling the fractional mean value formula for $0<s<1$ from \cite[Corollary 4.3]{MR2237634}:  for $h\in\R$ small,
\be\label{revised_mean_value}
w(x+h)=w(x)+h^s\frac{1}{\Gamma(1+s)}\hat d_L^sw(x+\theta h),
\ee
where the revised left-sided $RL$ $s$-order derivative $\hat d^s_L$ is defined in \eqref{revisied_RL} and $\theta\in\R$, say $\theta(h)$, depends upon $h$, and satisfies
\be
\lim_{h\to 0}\theta^s(h)=\frac{\Gamma(1+s)^2}{\Gamma(1+2s)}.
\ee
We note that 
\begin{align}
\hat d^s_Lw(x)&=\frac1{\Gamma(1-s)} \frac d{dx}\int_0^x \frac{w(t)-w(0)}{\fsp{x-t}^s}dt\notag \\
&= \frac1{\Gamma(1-s)} \frac d{dx}\int_0^x \frac{w(t)}{\fsp{x-t}^s}dt - \frac1{\Gamma(1-s)} \frac d{dx}\int_0^x \frac{w(0)}{\fsp{x-t}^s}dt \notag \\
&=d^s_L w(x)-w(0)\frac1{\Gamma(1-s)}\frac1{x^s}.\label{revised_mean_value2}
\end{align}
We also, by the definition of $\tilde w$, summarize the following 4 cases (w.l.o.g we only consider $\abs{h}<0.1$).
\begin{enumerate}[{Case} 1.]
\item Both $x$ and $(x+h)\in\R\setminus I$. In this case we have
\be
\abs{\tilde w(x+h)-\tilde w(x)}=0;
\ee
\item
$x\in I$, $h<0$ and $x+h<0$. In this case we have $0<x<-h$ and
\be
\abs{\tilde w(x+h)-\tilde w(x)}\leq \abs{\tilde w(x+h)- w(0)}+\abs{w(0)- w(x)}=\abs{w(0)}+\abs{w(0)- w(x)};
\ee
\item
$x\in I$, $h>0$ and $x+h>1$. In this case we have $0<1-x<h$ and
\be
\abs{\tilde w(x+h)-\tilde w(x)}\leq \abs{\tilde w(x+h)- w(1)}+\abs{w(1)- w(x)}=\abs{w(1)}+\abs{w(1)- w(x)};
\ee
\item
Both $x$ and $x+h\in I$. In this case we have 
\be
\abs{\tilde w(x+h)-\tilde w(x)}=\abs{ w(x+h)- w(x)}.
\ee
\end{enumerate}
We now claim \eqref{dengchaxiangjian_p1}. 
In any of above cases, we could deduce that 
\be\label{revised_mean_value3}
\norm{\tilde \tau_hw-\tilde w}_{L^1(\R)}\leq \abs{I^h\setminus I_h}\norm{w}_{L^\infty(\R)}+\norm{ \tau_hw- w}_{L^1(I_h)},
\ee
where $I_h=(\abs{h},1-\abs{h})$ and $I^h=(-\abs{h},1+\abs{h})$. Next, for $x$ and $x+h\in I$, we observe that
\begin{align*}
\abs{\tilde w(x+h)-\tilde w(x)}&\leq h^s\frac{1}{\Gamma(1+s)}\abs{\hat d_L^sw(x+\theta h)}\\
&\leq h^s\frac{1}{\Gamma(1+s)}\fmp{\abs{d^s_L w(x)}+w(0)\frac1{\Gamma(1-s)}\frac1{x^s}},
\end{align*}
where we used \eqref{revised_mean_value} and \eqref{revised_mean_value2}.\\\\
That is, we have
\be
\norm{\tilde w(x+h)-\tilde w(x)}_{L^1(I_h)}\leq h^s\frac{1}{\Gamma(1+s)}\fmp{TV^s(w,I)+\norm{w}_{L^\infty(\partial I)}[\Gamma(1-s)(1-s)]^{-1}},
\ee
and combining with \eqref{revised_mean_value3} we obtain 
\be
\norm{\tilde \tau_hw-\tilde w}_{L^1(\R)}\leq h^s\fmp{[\Gamma(1-s)(1-s)]^{-1}+[\Gamma(1+s)]^{-1}}\fmp{TV^s(w)+\norm{w}_{L^\infty(\partial I)}},
\ee
hence \eqref{dengchaxiangjian_p1}.
\end{proof}

Now we are ready to prove Theorem \ref{ATV_real_embedding} by using the approximation and slicing argument.
\begin{proof}[Proof of Theorem \ref{ATV_real_embedding}]
We shall only argue in the case that $N=2$, as the general case $N\geq 3$ can be obtained similarly. Let $\seqn{u_n}\subset BV^s(Q)$ be given be such that the assumption \eqref{image_condition_infty} holds. Let $\tilde u_n\in L^1(\R^2)$ be defined as
\be
\tilde u_n(x):=
\begin{cases}
u_n(x) &\text{ if }x\in Q\\
0 &\text{ if }x\in \R^2\setminus Q
\end{cases}
\ee
Then, by Proposition \ref{one_d_cuoweixiangjian}, we have, for any given $x_2\in \R$ and small $h\in \R$,
\be
\norm{\tau_h \tilde u_n(t,x_2)-\tilde u_n(t,x_2)}_{L^1(\R)}\leq h^s C_s [TV_{\ell^1}(u_n(t,x_2),I)+\norm{u_n}_{L^\infty(\partial Q)}],
\ee
where $C_s:=\fmp{[\Gamma(1-s)(1-s)]^{-1}+[\Gamma(1+s)]^{-1}}$. Thus, we have, for vector  $h\in \R^2$ with small
norm and parallel to the $x_1$ axis, 
\be
\norm{\tau_h \tilde u_n-\tilde u_n}_{L^1(\R^2)}\leq h^s C_s [TV_{\ell^1}^s(u_n)+\norm{u_n}_{L^\infty(\partial Q)}].
\ee
Similarly, for $h\in \R^2$ with small norm and parallel to the $x_2$ axis, 
\be
\norm{\tau_h \tilde u_n-\tilde u_n}_{L^1(\R^2)}\leq h^s C_s [TV_{\ell^1}^s(u_n)+\norm{u_n}_{L^\infty(\partial Q)}].
\ee
For a general direction vector $h\in\R^2$, we write $h=(h_1,h_2)$, and 
\begin{align*}
\norm{\tau_h \tilde u_n-\tilde u_n}_{L^1(\R^2)}&\leq \norm{\tau_{h_1} \tilde u_n-\tilde u_n}_{L^1(\R^2)}+\norm{\tau_{h_2} \tilde u_n-\tilde u_n}_{L^1(\R^2)}\\
&\leq 2h^s C_s C\fmp{\norm{u_n}_{BV^s(Q)}+\norm{u_n}_{L^\infty(\partial Q)}}.
\end{align*}
Thus, in view of Theorem \ref{lp_compactness}, there exists $u\in L^1(Q)$
 such that $\tilde u_n\to u$ strongly in $L^1(Q)$. That is, $u_n\to u$ strongly in $L^1(Q)$ and the
 proof is complete.
\end{proof}
\begin{corollary}
Given a sequence $\seqn{u_n}\subset BV^s(Q)$ such that the assumptions of Theorem \ref{ATV_real_embedding} hold, then, up to a subsequence, then there exists $u\in BV^s(Q)$ 
such that 
\be
u_n\to u\text{ strongly in }L^1\text{ and }\liminfn \,TV^s(u_n)\geq TV^s(u).
\ee
\end{corollary}
\begin{proof}
The proof is done by combining Theorems \ref{ATV_real_embedding} and \ref{weak_star_comp_s}.
\end{proof}
\subsection{Lower semi-continuity with respect to sequences of orders}
We first perform our analysis for the an-isotropic total variation $TV_{\ell^1}^r$, and then use the equivalence condition (Remark \ref{an_iso_equ}) 
to recover the case of isotropic total variation $TV^r$.

\subsubsection{Interpolation properties of an-isotropic total variation}\label{subsub_monotone}

We start by studying an monotonicity result of $TV_{\ell^1}^s$ with respect to the order $s\in(0,1)$, for function $u\in IM(Q)$ (recall space $IM(Q)$ from \eqref{image_function}).
\begin{proposition}[monotonicity of an-isotropic $TV_{\ell^1}^s$]\label{bergounioux2017fractional}
Let $0< s<t<1$ be given and $u\in SV_{\ell^1}^{t}(Q)\cap\, IM(Q)$. Then we have $u\in BV^s(Q)\cap IM(Q)$ and
\be
TV_{\ell^1}^{s}(u)\leq TV_{\ell^1}^{t}(u).
\ee
Especially, at $s=0$, we have
\be
\norm{u}_{L^1(Q)}\leq TV_{\ell^1}^r(u).
\ee
\end{proposition}
\begin{proof}
For a moment we assume that $0<s<t<1$. We again deal in one dimension first and we assume also that  $w\in BV^{t}(I)\cap C^\infty(I)$. We claim $w\in \mathbb I^{t}(L^1(I))$. Indeed, we have 
\be
d (\mathbb I^{1-s}w(x))=\frac d{dx}\int_0^x \frac{w(t)}{(x-t)^{s}}dt = d^{s} w(x).
\ee
Thus, by assumption we have
\be\label{cite_represent_frac1}
\abs{\mathbb I^{1-s}w}_{W^{1,1}(I)}=TV^{s}(w)<+\infty,
\ee
 and hence \eqref{inte_frac_cond} holds.\\\\
We next claim \eqref{bdy_frac_cond}. Indeed, we only need to consider the case $l=0$ and we observe that
\be
\abs{\mathbb I^{1-s}w(\e)}=\abs{\int_0^\e\frac{w(t)}{(\e-t)^s}dt}\leq \norm{w}_{L^\infty(\partial I)}\int_0^\e\frac{1}{(\e-t)^s}dt\leq  \norm{w}_{L^\infty(\partial I)}\e^{1-s}\to 0
\ee
which implies that
\be\label{cite_represent_frac2}
d^0 \mathbb I^{1-s}w(0)=0.
\ee
Thus, \eqref{cite_represent_frac1} and \eqref{cite_represent_frac2} allows us to use Theorem \ref{thm_MR1347689}, Assertion \ref{cite_represent_frac} 
to infer the existence of $f\in L^1(I)$ such that $w=\mathbb I^{t}[f]$. Moreover, in view of Theorem \ref{thm_MR1347689}, Assertion \ref{MR1347689T2_5}, we have
\be
w = \mathbb I^{t}[f] = \mathbb I^{s}\mathbb I^{t-s}[f].
\ee
We observe that
\be\label{point_norm_single}
\norm{d^{s}w}_{L^1(I)} = \norm{\mathbb I^{t-s}[f]}_{L^1(I)}\\
\leq \abs{\mathbb I^{t-s}}\norm{f}_{L^1(I)},
\ee
where by $\abs{\mathbb I^{t-s}}$ we denote the operator norm. Thus, in view of \eqref{eq_semigroup_frac_int},
 and setting $\delta=t-s>0$, we have
\be
\abs{\mathbb I^{\delta}}\leq \sup_{\norm{\vp}_{L^1}>0}\frac{\norm{\mathbb I^{\delta}\vp}_{L^1(a,b)}}{\norm{\vp}_{L^1(a,b)}}\leq (b-a)^{\delta}\frac{1}{\delta\Gamma(\delta)}=(b-a)^{\delta}\frac{1}{\Gamma(\delta+1)}\leq 1,
\ee
whenever $b-a\leq 1$. This, together with \eqref{point_norm_single}, gives
\be
\norm{d^{s}w}_{L^1(I)} \leq \norm{d^{t}w}_{L^1(I)},
\ee
as desired.\\\\
We next deal with the multi-dimensional case. We shall only write in details for $N=2$, as the case $N\geq 3$ 
is similar. Assume that $u\in BV^t(Q)\cap C^{\infty}(Q)$, and in view of Remark \ref{smooth_use_here}, we have
\be\label{real_momn_two2}
TV_{\ell^1}^t(u)=\int_Q\abs{\nabla^s u}dx=\int_0^1\int_0^1\abs{\partial^t_1 u(x)}dx_1dx_2+\int_0^1\int_0^1\abs{\partial^t_2 u(x)}dx_1dx_2.
\ee
Since $u\in C^\infty(Q)\cap BV^t(Q)$, we have, for each $x_2\in(0,1)$,
 the slice $w_{x_2}(x_1):=u(x_1,x_2)$ is well defined and belongs to $BV^t(I)$. Thus
\be
\int_0^1\abs{\partial^s_1u(x_1,x_2)}dx_1= TV_{\ell^1}^s(w_{x_2}(x_1))\leq TV_{\ell^1}^t(w_{x_2}(x_1))=\int_0^1\abs{\partial^t_1u(x_1,x_2)}dx_1.
\ee
Integrating over $x_2\in I$, we have 
\be\label{real_momn_two}
\int_0^1\int_0^1\abs{\partial^s_1 u(x)}dx_1dx_2\leq \int_0^1\int_0^1\abs{\partial^t_1 u(x)}dx_1dx_2.
\ee
Similarly we can show 
\be
\int_0^1\int_0^1\abs{\partial^s_2 u(x)}dx_1dx_2\leq \int_0^1\int_0^1\abs{\partial^t_2 u(x)}dx_1dx_2.
\ee
This, together with \eqref{real_momn_two} and \eqref{real_momn_two2}, gives
\be\label{smooth_result_tvts}
TV_{\ell^1}^s(u)\leq TV_{\ell^1}^t(u)\text{ for $u\in C^\infty(Q)\cap BV^t(Q)$}.
\ee 
Finally, assume $u\in SV_{\ell^1}^t(Q)$ only. From Definition \ref{strict_conv_BVr} we could obtain a sequence $\seqn{u_n}\subset BV^t(Q)\cap C^\infty(Q)$ such that $u_n\to u$ strongly in $L^1(Q)$ and 
\be\label{monon_approx}
TV_{\ell^1}^t(u_n)\to TV_{\ell^1}^t(u).
\ee
Then, in view of \eqref{smooth_result_tvts}, for each $u_n$ we have 
\be
TV_{\ell^1}^s(u_n)\leq TV_{\ell^1}^t(u_n).
\ee
Also, since $u_n\to u$ strongly in $L^1$, in view of Theorem \ref{weak_star_comp_s}, and together with \eqref{monon_approx}, we conclude that 
\be
TV_{\ell^1}^s(u)\leq \liminfn\, TV_{\ell^1}^s(u_n)\leq \limsup_{n\to\infty} TV_{\ell^1}^t(u_n)=TV_{\ell^1}^t(u).
\ee
In the end, take any $\vp\in C_c^\infty(Q;\R^2)$, we observe, in view of Lemma \ref{linftyds} and \eqref{scaled_divg}, that
\be
\liminf_{s\searrow 0}TV_{\ell^1}^s(u)\geq \liminf_{s\searrow 0}\int_Qu\,\divg^s\vp\,dx =\frac12\int_Q u\fsp{\vp_1+\vp_2}\,dx,
\ee
and hence, we conclude, by the arbitrariness of $\vp$, that 
\be
TV_{\ell^1}^s(u)\geq TV_{\ell^1}^0(u)= L^1(u),
\ee
which concludes the case that $s=0$, and hence the thesis.
\end{proof}
We next investigate the lower semi-continuity and compactness with respect to the
 order. 
\begin{proposition}\label{compact_lsc_s}
Given sequences $\seqn{s_n}\subset (0,1)$, $s_n\to s\in[0,1]$, and $\seqn{u_n}\subset L^1(Q)$ such that there exists $p\in(1,+\infty]$ satisfying
\be\label{sn_uniform_uppd}
\sup\flp{\norm{u_n}_{L^p}+TV^{s_n}(u_n):\,\,n\in\N}<+\infty,
\ee
then the following statements hold.
\begin{enumerate}[1.]
\item
There exists $u\in BV^s(Q)$ and, up to a subsequence, $u_n\wto u$ weakly in $L^p(Q)$, and 
\be\label{sn_uniform_lsc}
\liminf_{n\to\infty} TV^{s_n}(u_n)\geq TV^s(u).
\ee
\item
Assuming in additional that $u_n\in SV^{s_n}(Q)$ for each $n\in\N$, $s_n\to s\in(0,1]$, and $\seqn{u_n}\subset IM(Q)$, we have 
\be
u_n\to u\text{ strongly in }L^1(Q).
\ee
\end{enumerate}
\end{proposition}

\begin{proof}
Let $\seqn{u_n}\subset L^1(Q)$ be given, such that \eqref{sn_uniform_uppd} holds. Since $\norm{u_n}_{L^p(Q)}$ (with $p>1$) is uniformly bounded, 
there exists $u\in L^p(Q)$ such that, upon subsequence, 
\be\label{need_dunford_pettis}
u_n\wto u\text{ weakly in }L^p(Q).
\ee
We now prove Statement 1. In view of Lemma \ref{linftyds} we get that for any $\vp\in C_c^\infty(Q)$ and $s\in[0,1]$, it holds $\divg^s\vp\in L^{p'}(Q)$ (where $1/p +1/{p'}=1$) 
and hence, by the weak $L^p(Q)$-convergence in \eqref{need_dunford_pettis}, we have that 
\be
\limn \int_Q u_n\,\divg^s\vp\,dx= \int_Q u\,\divg^s\vp\,dx.
\ee
Statement 2 of Lemma \ref{linftyds} gives $\divg^{s_n}\vp\to \divg^s\vp$ a.e., while
Statement 1 of Lemma \ref{linftyds} gives
\[\sup_n \|\divg^{s_n}\vp\|_{L^\infty(Q)}<+\infty.\]
Thus, by dominated convergence theorem, $\divg^{s_n}\vp\to \divg^s\vp$ strongly
in $L^{p'}(Q)$.
By H\"older inequality, 
\begin{align*}
\limsup_{n\to\infty}\int_Q \abs{u_n}\abs{\divg^{s_n}\vp-\divg^s\vp}dx
\leq \sup_{n\in\N}\norm{u_n}_{L^p(Q)} \limsup_{n\to\infty} \norm{\divg^{s_n}\vp-\divg^s\vp}_{L^{p'}(Q)}=0.
\end{align*}
Therefore, we have
\begin{align}
&\limn \int_Q u_n \,\divg^{s_n}\vp\,dx\notag\\
&\geq\liminfn \int_Q u_n\,\divg^{s}\vp\,dx+\liminfn \int_Q u_n\fmp{\divg^{s_n}\vp-\divg^s\vp}dx\notag\\
&=\int_Q u\,\divg^s\vp\,dx,
\label{need_for_tv_case}
\end{align}
and hence
\be
\liminfn\, TV^{s_n}(u_n)\geq \liminfn\,\int_Q u_n\,\divg^{s_n}\vp\,dx 
=\limn \int_Q u_n\,\divg^{s_n}\vp\,dx= \int_Q u\,\divg^s\vp\,dx,
\ee
and \eqref{sn_uniform_lsc} follows by the arbitrariness of $\vp\in C_c^\infty(Q;\rn)$.\\\\
We next prove Statement 2. By Remark \ref{an_iso_equ} and \eqref{sn_uniform_uppd} we have that 
\be
\sup\flp{\norm{u_n}_{L^p}+TV_{\ell^1}^{s_n}(u_n):\,\,n\in\N}<+\infty,
\ee
Since $s_n\to s\in(0,1]$, in view of Proposition \ref{bergounioux2017fractional}, 
there exist a (sufficiently large) $N\in\N$ and $s_N<1$, such that 
\be
TV_{\ell^1}^{s_N/2}(u_n)\leq TV_{\ell^1}^{s_n}(u_n)+1.\qquad \text{for all }n\ge N.
\ee
Combined with the assumption that $\seqn{u_n}\subset IM(Q)$ allows us to use Theorem \ref{ATV_real_embedding} to infer Statement 2.
\end{proof}
The next theorem provides a connection between the lower order $TV^s$ to the higher order $TV^r$, where we recall again that $r=\ir+s$.
\begin{theorem}\label{liangbiankongzhizhongjian}
Let $r=\ir+s$ be given. There exists a constant $C_r>0$ such that 
\be
TV^{s}(u)\leq C_r\fmp{\norm{u}_{L^1(Q)}+TV^{r}(u)},
\ee
for all $u\in SV^r(Q)$.
\end{theorem}
Before we prove Theorem \ref{liangbiankongzhizhongjian}, we first prove an one dimension version.
\begin{lemma}\label{1d_approx_up}
Let $I=(0,1)$ and $r=\ir+s$ be given. There exists a constant $C_r>0$ such that 
\be
\sum_{l=0}^{\ir-1}TV^{s+l}(w)\leq C_r\fmp{\norm{w}_{L^1(I)}+TV^{r}(w)},
\ee
for all $w\in SV^r(I)$.
\end{lemma}
\begin{proof}
We deal with the case $\ir=1$ first. That is, we have $r=1+s$.\\\\
Assume no such $C_r$ exists, i.e., there exists a sequence $\seqn{w_n}\subset SV^{1+s}$ such that, for each $n\in\N$,
\be
TV^s(w_n)=1\text{ and }\norm{w_n}_{L^1(I)}+TV^{1+s}(w_n)<1/n.
\ee
In view of Theorem \ref{approx_smooth}, we may as well assume that $\seqn{w_n}\subset C^\infty(I)\cap BV^{1+s}(I)$ and we may write
\be\label{where_contra}
\norm{d^sw_n}_{L^1(I)}\geq1/2\text{ and }\norm{w_n}_{L^1(I)}+\norm{d^{1+s}w_n}_{L^1(I)}<2/n.
\ee
Thus, we have that $\seqn{w_n}\subset W^{1+s}(I)$ and 
\be\label{strong_zero}
w_n\to 0\text{ and }d^{1+s}w_n\to 0\text{ strongly in }L^1(I).
\ee
Recall from Theorem \ref{thm_MR3144452}, we may write that, for each $w_n$, 
\be
w_n(t)= \frac{c_{0,n}}{\Gamma(s)}t^{s-1}+ \frac{c_{1,n}}{\Gamma(s+1)}t^{s}+\mathbb I^{1+s} \phi_n(t),\,\, t\in I\,\,a.e.,
\ee
and 
\be
d^{1+s} w_n(t)=\phi_n(t).
\ee
Thus, in view of \eqref{strong_zero} we have 
\be
\phi_n\to 0\text{ strongly in }L^1(I),
\ee
and together with \eqref{eq_semigroup_frac_int}, we have, for any $0\leq s'\leq s$,
\be\label{shangyidian0}
\norm{\mathbb I^{1+s'}\phi_n(t)}_{L^1(I)}\leq \frac1{\Gamma(2+s')}\norm{\phi_n}_{L^1(I)}\to 0.
\ee
We next claim that 
\be\label{shangyidian1}
c_{0,n}\to 0\text{ and }c_{1,n}\to 0.
\ee
By the mean value theorem we have
\be
(\mathbb I^{1-s}w)(t_0)=\int_I (\mathbb I^{1-s}w)(l)dl
\ee
Since $w\in W^{1+s}(I)$, we have $(\mathbb I^{1-s}w)$ is absolutely continuous, hence
\be
(\mathbb I^{1-s}w)(t)=(\mathbb I^{1-s}w)(t_0)+\int_{t_0}^t d^1(\mathbb I^{1-s}w)(l)dl = (\mathbb I^{1-s}w)(t_0)+\int_{t_0}^t (d^sw)(l)dl.
\ee
Thus, for any $t\in[0,1]$,
\be
\abs{(\mathbb I^{1-s}w)(t)}\leq \norm{\mathbb I^{1-s}w}_{L^1(I)}+\norm{d^s u}_{L^1}\leq \frac{1}{\Gamma(2-s)}\norm{w}_{L^1(I)}+\norm{d^s w}_{L^1(I)},
\ee
where at the last inequality we used \eqref{eq_semigroup_frac_int}. This, and together with \eqref{d_representable}, we obtain that 
\be
\abs{c_{0,n}}=\abs{(\mathbb I^{1-s}w_n)(0)}\leq \frac{1}{\Gamma(2-s)}\norm{w_n}_{L^1(I)}+\norm{d^s w_n}_{L^1}.
\ee
Similarly, we may show that 
\be
\abs{c_{1,n}}=\abs{(d(\mathbb I^{1-s})w_n)(0)}\leq \norm{d^sw_n}_{L^1(I)}+\norm{d^{1+s} w_n}_{L^1}.
\ee
Thus, in view of \eqref{where_contra}, we have 
\be\label{const_upper_bdd}
\sup\flp{\abs{c_{0,n}}+\abs{c_{1,n}}:\,\, n\in\N}<+\infty.
\ee
We also notice that 
\begin{align}
&\norm{\frac{c_{0,n}}{\Gamma(s)}t^{s-1}+ \frac{c_{1,n}}{\Gamma(s+1)}t^{s}}_{L^1(I)} \\&=\norm{\frac{c_{0,n}}{\Gamma(s)}t^{s-1}+ \frac{c_{1,n}}{\Gamma(s+1)}t^{s}+\mathbb I^{1+s}\phi_n(t)-\mathbb I^{1+s}\phi_n(t)}_{L^1(I)}\notag\\
&=\norm{w_n-\mathbb I^{1+s}\phi_n(t)}_{L^1(I)}\notag\\
&\leq \norm{w_n}_{L^1(I)}+\norm{\mathbb I^{1+s}\phi_n(t)}_{L^1(I)}\to 0,
\label{jiajiajia_l1}
\end{align}
which, combined with \eqref{const_upper_bdd}, implies 
\[c_{0,n}\to 0\text{ and }c_{1,n}\to 0,\]
and hence \eqref{shangyidian1}.
Then, in view of Lemma \ref{power_function_s}, we have
\be
d^s t^{s-1}=0 \text{ and }d^s t^s = \Gamma(s+1).
\ee
Next, in view of Definition \ref{frac_represent_I} and Theorem \ref{thm_MR1347689}, Assertion \ref{cite_represent_frac}, 
we have $\phi\in \mathbb I^s(L^1(I))$. Combined with  Statement 2 of Theorem \ref{thm_MR1347689}, Assertion \ref{MR1347689T2_5}, we obtain 
\be
d^s \mathbb I^{1+s}\phi_n(t)=\mathbb I^1\phi_n(t).
\ee
Thus, we have
\be
d^s w_n = c_{1,n}+\mathbb I^1\phi_n(t),
\ee
and together with \eqref{shangyidian0} and \eqref{shangyidian1}, we conclude that 
\be\label{shangyidian2}
\limsup_{n\to 0}\norm{d^{s} w_n}_{L^1(I)}\leq \abs{c_{1,n}}+\norm{\mathbb I^1\phi_n}_{L^1(I)}\to 0,
\ee
which contradicts \eqref{where_contra}.\\\\
Now we consider the general case $\ir\in\N$. In view of \eqref{repre_AC}, we write 
\be
w_n(t)=\sum_{l=0}^\ir \frac{c_{l,n}}{\Gamma(s+l)}t^{s-1+l}+\mathbb I^r\phi_n(t),\,\, t\in I\,\,a.e..
\ee
Similarly to \eqref{jiajiajia_l1}, we may show that  
\be
\abs{c_{0,n}}\leq \frac{1}{\Gamma(2-s)}\norm{w_n}_{L^1(I)}+\norm{d^s w_n}_{L^1(I)}
\ee
and
\be
\abs{c_{l,n}}\leq \norm{d^{s+l-1}w_n}_{L^1(I)}+\norm{d^{s+l} w_n}_{L^1(I)},
\ee
for $l=1,\ldots, \ir-1$, and further deduce that
\be
c_{l,n}\to 0\text{ for }i=0,\ldots, \ir.
\ee
Note that, for each $i=0,\ldots, \ir-1$,
\begin{enumerate}[1.]
\item
$d^{l+s} t^{s-1+j}=0$, for $j=0,\ldots, l$;
\item
$d^{l+s} t^{l+s} = \Gamma(s+1)$;
\item
$d^{l+s} t^{l+s+j} = \frac{\Gamma(l+s+j+1)}{\Gamma(j+1)}t^j$, for $j=1,\ldots,\ir-1-l$,
\item
$d^{l+s}(\mathbb I^r[\phi_n](t))=\mathbb I^{\ir-l}[\phi_n](t)$, a.e., $t\in I$.
\end{enumerate}
Therefore, we obtain that 
\be
\limsup_{n\to 0}\norm{d^{l+s} w_n}_{L^1(I)}\leq \sum_{j=0}^{\ir-1-l}\abs{c_{l+s+j,n}}+\norm{\mathbb I^{\ir-l}[\phi_n]}_{L^1(I)}\to 0,
\ee
which is a contradiction, and hence we conclude our thesis.
\end{proof}

\begin{proof}[Proof of Theorem \ref{liangbiankongzhizhongjian}]
We again only deal with the case $N=2$, as the case $N\geq 3$ is similar. We first show that 
\be\label{new_ATV_conclude}
TV_{\ell^1}^{s}(u)\leq C_r\fmp{\norm{u}_{L^1(Q)}+TV_{\ell^1}^{r}(u)}.
\ee
We assume for a moment that $u\in BV^r(Q)\cap C^\infty(Q)$. We start with the case $\ir=1$. That is, $r=1+s$. In view of Remark \ref{smooth_use_here}, we have
\be\label{eq_s_order}
TV_{\ell^1}^s(u)=\int_Q\abs{\partial_1^su}dx+\int_Q\abs{\partial_2^su}dx
\ee
and
\be\label{eq_1s_order}
TV_{\ell^1}^{1+s}(u)=\int_Q\abs{\partial_1^{1+s}u}dx+\int_Q\abs{\partial_2^{1+s}u}dx+\int_Q\abs{\partial_1^{s}\partial_2u}dx+\int_Q\abs{\partial_2^{s}\partial_1u}dx.
\ee
Let $w(t):=u(t,x_2)$, with a fixed $x_2\in I$. Then, by Lemma \ref{1d_approx_up}, we have
\be
TV_{\ell^1}^s(w(t))\leq C_r\fmp{\norm{w(t)}_{L^1(I)}+TV_{\ell^1}^{1+s}(w(t))}. 
\ee
That is, 
\be
\int_0^1\abs{\partial_1^su(x_1,x_2)}dx_1\leq C_r\fmp{\int_0^1 \abs{u(x_1,x_2)}dx_1+\int_0^1\abs{\partial_1^{1+s}u(x_1,x_2)}dx_1},
\ee
and hence
\begin{align*}
\int_Q\abs{\partial_1^su}dx &= \int_0^1\int_0^1\abs{\partial_1^su(x_1,x_2)}dx_1dx_2\\
&\leq C_r\fmp{\int_0^1\int_0^1 \abs{u(x_1,x_2)}dx_1dx_2+\int_0^1\int_0^1\abs{\partial_1^{1+s}u(x_1,x_2)}dx_1dx_2} \\
&=C_r\fmp{ \norm{u}_{L^1(Q)}+\int_Q \abs{\partial_1^{1+s}u}dx}.
\end{align*}
We may analogously prove 
\be
\int_Q\abs{\partial_2^su}dx \leq C_r\fmp{ \norm{u}_{L^1(Q)}+\int_Q \abs{\partial_2^{1+s}u}dx},
\ee 
and, in view of \eqref{eq_s_order} and \eqref{eq_1s_order},
\be\label{eq_1s_order_smooth}
TV_{\ell^1}^s(u)\leq C_r\fmp{\norm{u}_{L^1(Q)}+TV_{\ell^1}^{1+s}(u)},\text{ for each $u\in C^\infty(Q)\cap BV^r(Q)$. }
\ee
To conclude, we take an approximating sequence $\seqn{u_n}\subset C^\infty(Q)\cap BV^{1+s}(Q)$ such that $u_n\to u$ strongly in 
$L^1(Q)$ and $TV_{\ell^1}^{1+s}(u_n)\to TV_{\ell^1}^{1+s}(u)$. The former implies that 
\be
\liminfn TV_{\ell^1}^s(u_n)\geq TV_{\ell^1}^s(u),
\ee
and together with \eqref{eq_1s_order_smooth} we conclude that, for $u\in SV^{1+s}(Q)$,
\begin{align}
TV_{\ell^1}^s(u)&\leq \liminfn TV_{\ell^1}^s(u_n)\leq \liminf_{n\to\infty}2C_r\fmp{\norm{u_n}_{L^1(Q)}+TV_{\ell^1}^{1+s}(u_n)} \\
&\leq \limsup_{n\to\infty}2C_r\fmp{\norm{u_n}_{L^1(Q)}+TV_{\ell^1}^{1+s}(u_n)} = 2C_r\fmp{\norm{u}_{L^1(Q)}+TV_{\ell^1}^{1+s}(u)},
\end{align}
as desired.\\\\
Now we assume $\ir=2$. Similarly to the case $\ir=1$, we observe that 
\be
\int_Q\abs{\partial_1^su}dx \leq C_r\fmp{ \norm{u}_{L^1(Q)}+\int_Q \abs{\partial_1^{2+s}u}dx},
\ee
and
\be
\int_Q\abs{\partial_2^su}dx \leq C_r\fmp{ \norm{u}_{L^1(Q)}+\int_Q \abs{\partial_2^{2+s}u}dx}.
\ee
That is, we conclude that
\begin{align*}
&TV_{\ell^1}^s(u)\leq 2C_r\fmp{ \norm{u}_{L^1(Q)}+\int_Q \abs{\partial_1^{2+s}u}dx+\int_Q \abs{\partial_2^{2+s}u}dx}\\
&\leq 2C_r\fmp{ \norm{u}_{L^1(Q)}+\int_Q \abs{\partial_1^{2+s}u}dx+\int_Q \abs{\partial_2^{2+s}u}dx+\int_Q \abs{\partial_1^{1+s}\partial_2u}dx+\int_Q \abs{\partial_2^{1+s}\partial_1u}dx}\\
&=2C_r\fmp{\norm{u}_{L^1(Q)}+TV_{\ell^1}^r(u)}.
\end{align*}
For the general case that $\ir\in \N$, we shall always have, in view of Lemma \ref{1d_approx_up}, that 
\be
TV_{\ell^1}^s(u)\leq 2C_r\fmp{ \norm{u}_{L^1(Q)}+\int_Q \abs{\partial_1^{\ir+s}u}dx+\int_Q \abs{\partial_2^{\ir+s}u}dx},
\ee
and we conclude \eqref{new_ATV_conclude} as the right hand side is bounded by $\norm{u}_{L^1(Q)}+TV_{\ell^1}^r(u)$. Finally, we conclude our thesis by invoking again the equivalent condition \eqref{equ_p_1_p}.
\end{proof}
We close Section \ref{subsub_monotone} by proving that the constant $C_r$ from
Theorem \ref{liangbiankongzhizhongjian} can be taken independent of $r$. 
\begin{proposition}\label{zhendexiao_le}
Let $k\in\N$ be given. Then there exists $C>0$ such that for all $s\in(0,1)$,
\be
TV^{s}(u)\leq C\fmp{\norm{u}_{L^1(Q)}+TV^{k+s}(u)}
\ee
for  $u\in SV^{k+s}(Q)$.
\end{proposition}
\begin{proof}
We only proof this proposition for case of dimension one, i.e., $N=1$. The case in which $N\geq 2$ can be obtained from
the one dimensional result, and the arguments from Theorem \ref{liangbiankongzhizhongjian}. \\\\
Let $w\in BV^{k+s}(I)$ be given. In view of Lemma \ref{liangbiankongzhizhongjian} we have, for each $s\in (0,1)$, a constant $C_s>0$
(depending on $s$) such that 
\be\label{unif_upper_bdd}
\sum_{l=0}^{\ir-1}TV^{s+l}(w)\leq C_s\fmp{\norm{w}_{L^1(I)}+TV^{\ir+s}(w)}
\ee
for all $w\in L^1(I)$. \\\\
We shall only deal with the case $\ir=1$, as the case $\ir>1$ can be treated analogously. Suppose \eqref{unif_upper_bdd} fails, i.e.
there exist sequences $\seqn{w_n}\subset L^1(I)$ and $\seqn{s_n}\subset (0,1)$ such that 
\be
TV^{s_n}(w_n)=1 \qquad\text{ and }\qquad\norm{w_n}_{L^1(I)}+TV^{1+s_n}(w_n)<1/n.
\ee
In view of \eqref{const_upper_bdd}, we have
\begin{align}
\abs{c_{0,n}}+\abs{c_{1,n}}&\leq \fmp{\frac1{\Gamma(2-s_n)}+1}\fmp{\norm{w_n}_{L^1}+TV^{s_n}(w_n)+TV^{1+s_n}(w_n)}\notag\\
&\leq \frac1{\Gamma(2-s_n)}+1+\frac1n\leq 2+\frac1n.
\end{align}
Hence, there exist $c_0$ and $c_1$ such that $c_{0,n}\to c_0$ and $c_{1,n}\to c_1$. Then, we may reach the contradiction by using the same arguments from \eqref{shangyidian1} to \eqref{shangyidian2}.
\end{proof}
\subsubsection{Compact embedding and lower semi-continuity}\label{subsubsec_compact}
We start again with a result on the an-isotropic total variation. Recall the image space $IM(Q)$ from \eqref{image_function}.
\begin{proposition}\label{compact_lsc_r}
Given sequences $\seqn{r_n}\subset \R^+$ and $\seqn{u_n}\subset L^1(Q)$ such that $r_n\to r\in\R^+$ and, for some $p>1$,
\be\label{rn_uniform_uppd}
\sup\flp{\norm{u_n}_{L^p(Q)}+TV_{\ell^1}^{r_n}(u_n):\,\,n\in\N}<+\infty,
\ee
then, the following statement hold.
\begin{enumerate}[1.]
\item
There exists $u\in BV^r(Q)$, such that, up to a subsequence, $u_n\wto u$ in $L^p(Q)$ and 
\be
\liminf_{n\to\infty} TV_{\ell^1}^{r_n}(u_n)\geq TV_{\ell^1}^r(u).
\ee
\item
Assuming in addition that ${u_n}\subset IM(Q)\cap SV^{r_n}(Q)$ and $r_n\to r>0$, then
\be
u_n\to u\text{ strongly in }L^1(Q).
\ee
\end{enumerate}
\end{proposition}
\begin{proof}
Write $r_n=\lfloor r_n\rfloor+s_n$ for $s_n\in[0,1)$. By \eqref{rn_uniform_uppd}, there exists $u\in L^p(Q)$ such that, up to a subsequence,
\be\label{weak_r_lp}
u_n\wto u\text{ weakly in }L^p(Q).
\ee
Then Statement 1 can be proved by using the same arguments from Proposition \ref{compact_lsc_s}. \\\\
We next prove Statement 2. We only study the case of $1\leq r_n\leq 2$, or equivalently $\ir=1$, as the case in which $\ir\geq 2$ can be dealt analogously. Again by applying Proposition \ref{zhendexiao_le} we have 
\begin{align}
\sup_n\norm{u_n}_{BV^{s_n}(Q)} &\leq C\sup_n\norm{u_n}_{BV^{r_n}(Q)}
\leq \sup_n\{\norm{u_n}_{L^p(Q)}+TV_{\ell^1}^{r_n}(u_n)\}=:M<+\infty,
\label{jiangweigongji}
\end{align}
where $C>0$, obtained from Proposition \ref{zhendexiao_le}, is a constant independent of $r_n$. Assume in addition that there exists $\e>0$ such that $\seqn{s_n}\subset[\e,1]$. Then, in view of Proposition \ref{compact_lsc_s}, there exists $u\in L^1(Q)$ 
such that, up to a subsequence, that $u_n\to u$ strongly in $L^1(Q)$, which gives Statement 2.\\\\
We now deal with the situation that $s_n\searrow 0$, that is, $r_n\searrow 1$. In this case, although \eqref{jiangweigongji} still holds, Proposition \ref{compact_lsc_s} does not produce a subsequence
strongly converging in $L^1(Q)$. We proceed by using Theorem \ref{approx_smooth} to relax $u_n$, for each $n\in\N$, such that $u_n\in BV^{1+s_n}(Q)\cap C^\infty(Q)$. Hence, we have, for arbitrary $\vp\in C_c^\infty(Q,\R^2)$, that (recall \eqref{RLFODr})
\be
-\int_Q \nabla^{s_n} u_n\,\divg \vp\,dx=\int_Q u_n[\divg^{s_n}\divg] \vp\,dx=\int_Q u_n\divg^{1+s_n}\vp\,dx\leq TV_{\ell^1}^{1+s_n}(u_n),
\ee
which implies 
\be\label{ds_tv_eq}
TV_{\ell^1}(\nabla^{s_n}u_n)\leq TV_{\ell^1}^{1+s_n}(u_n)\leq M<+\infty.
\ee
Moreover, by Proposition \ref{zhendexiao_le}, 
\be\label{ds_tv_eq2}
\norm{\nabla^{s_n}u_n}_{L^1(Q)} = TV_{\ell^1}^{s_n}(u_n)\leq C \norm{u_n}_{BV^{1+s_n}(Q)}\leq M<+\infty.
\ee
We claim 
\be\label{un_dsn_go}
\limsup_{n\to\infty}\norm{u_n-\nabla^{s_n}u_n}_{L^1(Q)}=0.
\ee
We start from the one dimensional case, i.e. $Q=I=(0,1)$, and use $w_n$ to represent $u_n$. 
By Theorem \ref{thm_MR3144452}, for each $n\in\N$, there exists $\phi_n(t)\in L^1(I)$ such that
\be
w_n(t)= \frac{c_{0,n}}{\Gamma(s_n)}t^{s-1}+ \frac{c_{1,n}}{\Gamma(s_n+1)}t^{s_n}+\mathbb I^{1+s_n} \phi_n(t)
\ee
for a.e. $t$, and 
\be
d^{s_n} w_n(t) = c_{1,n} + \mathbb I^1\phi_n(t).
\ee
Next, in view of \eqref{r_int_frac_def}, we have, for each $n\in\N$ fixed, that 
\be
\mathbb I^{1+s}\phi_n(t) = \frac{1}{\Gamma(1+s)}\int_0^t \frac{\phi_n(z)}{\fsp{t-z}^{1-(1+s)}}dz=\frac{1}{\Gamma(1+s)}\int_0^t {\phi_n(z)}\fsp{t-z}^s dz.
\ee
That is, we have
\be
\lim_{t\to 0}\abs{\mathbb I^{1+s}\phi_n(t) }<+\infty\text{, for each }n\in\N.
\ee
However, $t^{s-1}\to +\infty$ as $t\to 0^+$. Hence, $c_{0,n}=0$ must hold, 
as the opposite gives $w_n(t)\to+\infty$ as $t\to 0$, which contradicts our assumption $w_n\in IM(Q)$. Thus, we have
\be
w_n(t)= \frac{c_{1,n}}{\Gamma(s+1)}t^{s}+\mathbb I^{1+s} \phi_n(t),\qquad \text{for a.e. }t.
\ee
Then, direct computation gives
\begin{align}\label{upper_un_dsn}
\begin{split}
\norm{w_n-d^{s_n}w_n}_{L^1(Q)} &= \norm{\frac{c_{1,n}}{\Gamma(s_n+1)}t^{s_n}-c_{1,n}+\mathbb I^{1+s_n}\phi_n-\mathbb I^1\phi_n}_{L^1(Q)}\\
&\leq \norm{\frac{c_{1,n}}{\Gamma(s_n+1)}t^{s_n}-c_{1,n}}_{L^1(Q)}+\norm{\mathbb I^{1+s_n}\phi_n-\mathbb I^1\phi_n}_{L^1(Q)}\\
&\leq \abs{c_{1,n}}\abs{\frac1{\Gamma(s_n+1)(s_n+1)}-1}+\abs{\mathbb I^{1+s_n}-\mathbb I^1}\norm{\phi_n}_{L^1(Q)}.
\end{split}
\end{align}
Moreover, using the same argument from the proof of \eqref{const_upper_bdd}, we have 
\be
\abs{c_{1,n}}+\norm{\phi_n}_{L^1(Q)}\leq 2\fmp{\norm{d^{s_n}w_n}_{L^1(Q)}+\norm{d^{s+n+1}w_n}_{L^1(Q)}}.
\ee
This, combined with \eqref{upper_un_dsn}, implies that 
\begin{align*}
&\|u_n-\nabla^{s_n}u_n\|_{L^1(Q)} = \int_0^1\norm{u_n\lfloor_{x_1}-d^{s_n}u_n\lfloor_{x_1}}_{L^1(Q)}dx_1\\
&\leq  2\fmp{\abs{\frac1{\Gamma(s_n+1)(s_n+1)}-1}+\abs{\mathbb I^{1+s_n}-\mathbb I^1}}
\int_0^1\fmp{\norm{d^{s_n}u_n\lfloor_{x_1}}_{L^1}+\norm{d^{s+n+1}u_n\lfloor_{x_1}}_{L^1}}dx_1\\
&\leq  2\fmp{\abs{\frac1{\Gamma(s_n+1)(s_n+1)}-1}+\abs{\mathbb I^{1+s_n}-\mathbb I^1}}\norm{u_n}_{BV^{r_n}(Q)}.
\end{align*}
By Theorem \ref{thm_MR1347689}, Assertion \ref{semigroup_frac_int}, we have $\abs{\mathbb I^{s_n+1}-\mathbb I^1}\to 0$, as $s_n\to 0$, and hence we conclude \eqref{un_dsn_go}. Next, by \eqref{ds_tv_eq} and \eqref{ds_tv_eq2}, and the compact embedding in standard $BV(Q)$ space, up to a subsequence, there exists $\bar u\in BV(Q)$, that 
\be
\nabla^{s_n}u_n\to \bar u\text{ strongly in }L^1(Q).
\ee
Hence, by \eqref{un_dsn_go}, we have 
\be
\norm{u_n-\bar u}_{L^1(Q)}\leq \norm{u_n- \nabla^{s_n}u_n}_{L^1(Q)}+\norm{\nabla^{s_n}u_n-\bar u}_{L^1(Q)}\to 0,
\ee
that is, we have $u_n\to \bar u$ strongly in $L^1(Q)$. 
Finally, in view of \eqref{weak_r_lp}, we have $u=\bar u$, and hence the thesis.
\end{proof}
We conclude this section by proving the second main result of this article.

\begin{theorem}[lower semi-continuity and compact embedding in $TV^r_\ellp$ seminorms]\label{compact_lsc_r_tv}
Given sequences $\seqn{r_n}\subset \R^+$, $\seqn{p_n}\subset [1,+\infty]$, and $\seqn{u_n}\subset L^1(Q)$ such that $r_n\to r\in\R^+\cup\flp{0}$ and $p_n\to p\in[1,+\infty]$,
 and there exists $q\in(1,+\infty]$ such that 
\be\label{uniform_tv_upperbdd}
\sup\flp{\norm{u_n}_{L^q(Q)}+TV^{r_n}(u_n):\,\,n\in\N}<+\infty,
\ee
then, the following statements hold.
\begin{enumerate}[1.]
\item
There exists $u\in BV^r(Q)$ such that, up to a subsequence, $u_n\wto u$ weakly in $L^q(Q)$ and 
\be
\liminf_{n\to\infty} TV_{\ell^{p_n}}^{r_n}(u_n)\geq TV_\ellp^r(u).
\ee
\item
Assuming in addition that ${u_n}\subset IM(Q)\cap SV^{r_n}(Q)$ and $r_n\to r>0$, we have 
\be
u_n\to u\text{ strongly in }L^1(Q).
\ee
\end{enumerate}
\end{theorem}
\begin{proof}
By applying Remark \ref{an_iso_equ}, we deduce that 
\eqref{uniform_tv_upperbdd} implies \eqref{rn_uniform_uppd}, 
and the thesis follows by combining \eqref{need_for_tv_case} and Theorem \ref{weak_star_comp_s}.
\end{proof}

\subsection{\emph{ROF} model with real order total variation}
In this subsection we equip the \emph{ROF} model, introduced in \eqref{intro_B_train_level2}, with the  $TV_{\ell^p}^r$ seminorm.\\\\
Given $r\in \R^+$ and $p\in[1,+\infty]$, we introduce the $ROF^r_{\ell^p}$ image processing framework 
\be
ROF_{\ell^p}^r(u):=\norm{u-u_\eta}_{L^2(Q)}^2+\alpha TV_{\ell^p}^{r}(u),
\ee
where we recall $u_\eta\in L^2(Q)$ is a given corrupted image.\\\\
The following proposition is a direct consequence of Theorem \ref{compact_lsc_r_tv}.
%\begin{proposition}[existence of minimizer]\label{tgv_r_exist}
%For $s\in(0,1)$ fixed, the space $BV^s(Q)$ is a Banach space and the minimizing problem, where $u_\eta\in L^2(Q)$ is given,
%\be
%\argmin\flp{\norm{u-u_\eta|}_{L^2(Q)}^2+{ATV^{s}(u)}:\,\,u\in BV^s(Q)},
%\ee
%admits a unique solution $u_s\in BV^s(Q)$.
%\end{proposition}
%\begin{proof}
%Let $\seqn{u_n}\subset BV^s(Q)$ be a minimizing sequence, then we have
%\be
%\norm{u_n}_{L^2(Q)}^2+ATV^s(u_n)\leq 2\norm{u_\eta}_{L^2(Q)}^2+ATV^s(u_\eta),
%\ee
%or in another word, we have
%\be
%\sup\flp{\norm{u_n}_{L^2(Q)}+\norm{u_n}_{BV^s(Q)}:\,\,n\in\N}<+\infty.
%\ee
%That is, up to a subsequence we have $u_n\wto u$ in $L^2$, and in view of Theorem \ref{weak_star_comp_s}, we are done.\\\\
%Moreover, if we in additional assume that $\norm{u_n}_{L^\infty}$ is bounded, then in view of Theorem \ref{real_embedding}, we may further assert that
%\be
%u_n\wto u\text{ strongly in }L^1.
%\ee
%This, and together with Theorem \ref{weak_star_comp_s}, we are done.
%\end{proof}
%
%
\begin{proposition}\label{tgv_r_exist}
Given $p\in[1,+\infty]$, $r\in\R^+$, and $u_\eta\in L^2(Q)$. Then we have that the minimization problems
\be\label{need_optimility_cond}
\argmin\flp{ROF_{\ell^p}^r(u):\,\,u\in BV^r(Q)}
\ee
and
\be\label{need_optimility_condss}
\argmin\flp{ROF_{\ell^p}^r(u):\,\,u\in SV^r(Q)},
\ee
admits a unique solution $u_r\in BV^r(Q)$ and $u_r\in SV^r(Q)$, respectively.
\end{proposition}
\begin{proof}
Let $\seqn{u_n}\subset L^1(Q)$ be an minimizing sequence. That is, we have
\be\label{tgv_r_exist_eq1}
\norm{u_n-u_\eta}_{L^2(Q)}^2+{TV_{\ell^p}^{r}(u_n)}\to m:=\inf\flp{\norm{u-u_\eta}_{L^2(Q)}^2+{TV_{\ell^p}^{r}(u)}:\,\,u\in BV^r(Q)}.
\ee
That is,
\be
\sup\flp{\norm{u_n}_{L^2(Q)}+TV_{\ell^p}^r(u_n):\,\, n\in\N}<+\infty.
\ee
This, combined with Theorem \ref{compact_lsc_r_tv} (with $p=2$ and $r_n=r\in\R^+$), gives that, up to a subsequence, 
\be\label{tgv_r_exist_eq2}
u_n\wto u_0\text{ weakly in }L^2(Q),
\ee
and
\be\label{tgv_r_exist_eq3}
\liminfn \,TV_{\ell^p}^r(u_n)\geq \liminfn\, TV_{\ell^p}^r(u_0).
\ee
Thus, by a standard argument from the Calculus of variation, we conclude that $u_0$ is the unique solution of \eqref{need_optimility_cond}.\\\\
We could similarly show that there exists $u_0\in BV^r(Q)$ belongs to \eqref{need_optimility_condss}. We only need to claim that $u_0\in SV^r(Q)$. Indeed, by combining \eqref{tgv_r_exist_eq1}, \eqref{tgv_r_exist_eq2}, and \eqref{tgv_r_exist_eq3}, we have that 
\be\label{tgv_r_exist_eq4}
\norm{u_n-u_\eta}_{L^2(Q)}\to \norm{u_0-u_\eta}_{L^2(Q)}\text{ and }TV_{\ell^p}^r(u_n)\to TV_{\ell^p}^r(u_0).
\ee
Since $\seqn{u_n}\subset SV^r(Q)$, we can obtain $\seqn{v_n}\subset SV^r(Q)$ such that 
\be
\norm{v_n-u_n}_{L^2(Q)}+\abs{TV_{\ell^p}^r(v_n)-TV_{\ell^p}^r(u_n)}\leq \frac1n.
\ee
This implies that $v_n\to u_0$ in $L^2$ strong and also
\be
\norm{v_n-u_\eta}_{L^2(Q)}+{TV_{\ell^p}^r(v_n)}\to \norm{u_0-u_\eta}_{L^2(Q)}^2+{TV_{\ell^p}^{r}(u_0)}.
\ee
Together with \eqref{tgv_r_exist_eq4}, we have 
\be
TV_{\ell^p}^r(v_n)\to TV_{\ell^p}^r(u_0),
\ee
and hence we have $u_0\in SV^r(Q)$.
\end{proof}
%
%Finally, we propose to use the primal-dual algorithm, studied in \cite{chambolle2011first}, to solve the minimizing problem \eqref{need_optimility_cond}. To do so, we recast \eqref{need_optimility_cond} as 
%\be\label{for_primal_dual}
%\min{\max\flp{-\alpha\fjp{u,\divg^r\vp}+\norm{u-u_\eta}_{L^2}^2-\delta_{V_{p^\ast}}(v):\,\, v\in C_c^\infty(Q;\R^2),\,\,u\in L^2(Q)}},
%\ee
%where $\delta_{V_{q^{\ast}}}$ denotes the indicator function of the convex set 
%\be
%V_{q}:=\flp{v\in C_c^\infty:\,\,\norm{v}_{L_{q}^\infty}\leq 1},
%\ee
%for $q\in[1,+\infty]$ with $\norm{v}_{L_q^\infty}:=\sup\flp{\abs{v(x)}_{\ell^q}:\,\, x\in Q}$. Then, the framework \eqref{for_primal_dual} allows us to incorporate the primal-dual algorithm to obtain the minimizer $u_r$ in \eqref{need_optimility_cond}.
%
%

\section{Special case with zero boundary conditions and center-sided operations}\label{image_application_sec}
\subsection{Equivalence between space $SV$ and $BV$ in zero boundary conditions}
\begin{theorem}\label{smooth_approx_bdy_condition}
Let $r\in\R^+$ and $u\in BV^r(Q)$ be given such that the zero boundary is satisfied, i.e., 
\be
u\in BV^r(Q)\cap H_0^{\ir+1}(Q).
\ee

Then there exists a sequence $\seqn{u_n}\subset C^\infty(Q)\cap BV^r(Q)$ such that 
\begin{equation}
u_n \to u\text{ strongly in }L^1(Q)\text{ and }\limn TV_\ellp^r(u_n) = TV_\ellp^r(u).
\label{convergence Lp}
\end{equation}
\end{theorem}

\begin{proof}
We first prove a regularity result:
\begin{equation}
\frac{\partial^\ir}{\partial x_1^{\alpha_1} \cdots \partial x_n^{\alpha_n} }D_{L x_j}^s u \in L^2(Q),
\label{L2 of derivatives}
\end{equation}
for all $j=1,\cdots,n$ and indices $\alpha_1,\cdots,\alpha_n$ such that $\sum_{k=1}^n\alpha_k=\ir$.\\\\
Let $r>0 $ be given, and assume that $r\in \R^+\setminus \N$. Let $s:=r-\ir$, and by integration by parts, for any $\vp\in C_c^\infty$, there holds
\begin{align*}
TV^r(u)&=\sup\bigg\{ \int_Q u \,\divg^s (\divg^\ir \varphi) dx: |\varphi|\le1, \varphi\in C_c^\infty(Q)\bigg\}\\
&=\sup\bigg\{ (-1)^{\ir+1}\int_Q D^\ir (D^su) \cdot \varphi dx: |\varphi|\le1, \varphi\in C_c^\infty(Q)\bigg\}. 
\end{align*}
Note that $u\in H_0^{\ir+1}(Q)$ gives $D^\ir (D^su)\in L^2(Q)$: let $x=(x_1,\cdots,x_n)$, by change of variable
$y=x_j-t$, we get
\begin{align*}
&\frac{\partial^\ir}{\partial x_1^{\alpha_1} \cdots \partial x_n^{\alpha_n} }D_{L x_j}^s u(x)\\
&=\frac{\partial^\ir}{\partial x_1^{\alpha_1} \cdots \partial x_n^{\alpha_n} }
\frac{\partial}{\partial x_j} \int_0^{x_j}\frac{ u(x_1,\cdots,x_{j-1},t,x_{j+1},\cdots,x_{n}) }{(x_j-t)^s} dy\\
&=\frac{\partial^\ir}{\partial x_1^{\alpha_1} \cdots \partial x_n^{\alpha_n} }
\frac{\partial}{\partial x_j}\bigg(\int_0^{x_j}\frac{ u(x_1,\cdots,x_{j-1},x_j-y,x_{j+1},\cdots,x_{n}) }{y^s} dy  \bigg)\\
&=\frac{\partial^\ir}{\partial x_1^{\alpha_1} \cdots \partial x_n^{\alpha_n} }
\bigg(\int_0^{x_j}\frac{ \frac{\partial}{\partial x_j}u(x_1,\cdots,x_{j-1},x_j-y,x_{j+1},\cdots,x_{n}) }{y^s} dy \\
&\qquad + \frac{ \frac{\partial}{\partial x_j}u(x_1,\cdots,x_{j-1},x_j-y,x_{j+1},\cdots,x_{n}) }{y^s}\bigg|_{y=x_j}\bigg),
\end{align*}
where the boundary term
\[\frac{ \frac{\partial}{\partial x_j}u(x_1,\cdots,x_{j-1},x_j-y,x_{j+1},\cdots,x_{n}) }{y^s}\bigg|_{y=x_j}\equiv 0 \]
due to $u\in H_0^{\ir+1}(Q)$. Thus
\begin{align*}
&\frac{\partial^\ir}{\partial x_1^{\alpha_1} \cdots \partial x_n^{\alpha_n} }D_{L x_j}^s u(x)\\
&=\frac{\partial^\ir}{\partial x_1^{\alpha_1} \cdots \partial x_n^{\alpha_n} }
\int_0^{x_j}\frac{ \frac{\partial}{\partial x_j}u(x_1,\cdots,x_{j-1},x_j-y,x_{j+1},\cdots,x_{n}) }{y^s} dy,
\end{align*}
and by taking all the derivatives, we get
\begin{align}
&\frac{\partial^\ir}{\partial x_1^{\alpha_1} \cdots \partial x_n^{\alpha_n} }D_{L x_j}^s u(x)\\
&=
\int_0^{x_j}y^{-s}
\frac{\partial^\ir}{\partial x_1^{\alpha_1} \cdots \partial x_n^{\alpha_n} }
 \frac{\partial}{\partial x_j}u(x_1,\cdots,x_{j-1},x_j-y,x_{j+1},\cdots,x_{n})  dy \notag\\
& =\int_{\mathbb{R}}y^{-s}\chi_{(0,1)}(y)
\frac{\partial^\ir}{\partial x_1^{\alpha_1} \cdots \partial x_n^{\alpha_n} }
 \frac{\partial}{\partial x_j}u(x_1,\cdots,x_{j-1},x_j-y,x_{j+1},\cdots,x_{n}) \chi_{(0,1)}(x_j-y) dy\notag\\
 &= \bigg[(|\cdot|^{-s}\chi_{(0,1)}) * \Big(\chi_{(0,1)}\frac{\partial^\ir}{\partial x_1^{\alpha_1} \cdots \partial x_n^{\alpha_n} }
 \frac{\partial u}{\partial x_j} \Big)\bigg](x).
\end{align}
Since $u\in H_0^{\ir+1}(Q)$, we get 
\[\chi_{(0,1)}\frac{\partial^\ir}{\partial x_1^{\alpha_1} \cdots \partial x_n^{\alpha_n} }
 \frac{\partial u}{\partial x_j} \in L^2(Q), \]
and since clearly $|\cdot|^{-s}\chi_{(0,1)}\in L^1(0,1)$, 
 \eqref{L2 of derivatives} is proven.\\\\
Now we turn our attention to \eqref{convergence Lp}.
We need only to discuss the left-sided RL derivative, as the discussion for the right-sided RL derivative
is completely analogous.\\\\
We first present the 1D case. 
Note that
 \begin{align}
\frac{d^\ir}{dx^\ir} D_L^s u(x) &=  \frac{1}{\Gamma(1-s)} \frac{d^\ir}{dx^\ir} \bigg(\frac{d}{dx}\int_0^x u(t)(x-t)^{-s} dt\bigg)\notag\\
&\overset{u^{(k)}(0)=0,\ k=0,\cdots,\ir}=\frac{1}{\Gamma(1-s)} \int_0^x u^{(\ir+1)}(t)(x-t)^{-r} dt\notag\\
 &= \frac{1}{\Gamma(1-s)}\int_{-\infty}^{+\infty} u^{(\ir+1)}(t)\chi_{(0,1)}(t)  (x-t)^{-s}\chi_{(0,1)}(x-t) dt\notag\\
 &=\frac{1}{\Gamma(1-s)} \Big[(u^{(\ir+1)}\chi_{(0,1)})*(|\cdot|^{-s}\chi_{(0,1)}) \Big](x)\label{1st form},
 \end{align}
 which implies that 
 \begin{align}
\bigg\|\frac{d^\ir}{dx^\ir} D_L^s u\bigg\|_{L^1(I)} 
 = \frac{1}{\Gamma(1-s)} \Big\|(u^{(\ir+1)}\chi_{(0,1)})*(|\cdot|^{-s}\chi_{(0,1)}) \Big\|_{L^1(I)}.
  \end{align}
  Let $u_n\subseteq C_c^\infty(I)$ be a sequence of smooth functions such that $u_n\to u$
  strongly in $H_0^{\ir+1}(I)$. Then 
  \begin{align*}
   \bigg\|\frac{d^\ir}{dx^\ir} D_L^s u_n\bigg\|_{L^1(I)} & \le 
   \bigg\|\frac{d^\ir}{dx^\ir} D_L^s u\bigg\|_{L^1(I)} + \bigg\|\frac{d^\ir}{dx^\ir} D_L^s (u-u_n)\bigg\|_{L^1(I)} ,
  \end{align*}
where
\begin{align*}
&\bigg\|\frac{d^\ir}{dx^\ir} D_L^s (u-u_n)\bigg\|_{L^1(I)} \\
&= \frac{1}{\Gamma(1-s)} \Big\|((u-u_n)^{(\ir+1)}\chi_{(0,1)})*(|\cdot|^{-s}\chi_{(0,1)}) \Big\|_{L^1(I)} \\
&\le \frac{1}{\Gamma(1-s)} \Big\|((u-u_n)^{(\ir+1)}\chi_{(0,1)})\Big\|_{L^1(I)} \Big\|(|\cdot|^{-s}\chi_{(0,1)}) \Big\|_{L^1(I)}
\to 0,
\end{align*}
since 
\[\Big\|((u-u_n)^{(\ir+1)}\chi_{(0,1)})\Big\|_{L^1(I)} \to 0\]
as $u_n\to u$ strongly in $H_0^{\ir+1}(I)$, and clearly 
\[\Big\|(|\cdot|^{-s}\chi_{(0,1)}) \Big\|_{L^1(I)} = \frac{1}{1-s}. \]
 Now we deal with the case that $N\geq 2$. The proof is essentially analogous to the 1D case: by integration by parts,
 \begin{align}
 TV^r(u)&=\sup_{|\varphi|\le 1,\ \varphi\in C_c^\infty(Q)} 
 \int_Q u\, \divg^s (\divg^\ir \varphi) dx \\
 &= 
 \sup_{|\varphi|\le 1,\ \varphi\in C_c^\infty(Q)}(-1)^{\ir+1}\int_Q D^\ir (D^su) \cdot \varphi \,dx,
 \end{align}
 and, in view of Remark \ref{smooth_use_here}, it suffices to discuss each component of $D^\ir (D^su)$ separately. That is,
 we will show the existence of a sequence $u_n\subseteq C_c^\infty(Q)$ such that
 \[ \bigg\|\frac{\partial^\ir}{\partial x_1^{\alpha_1} \cdots \partial x_n^{\alpha_n} }D_{L x_j}^s u_n \bigg\|_{L^1(Q)}
 \le \bigg\|\frac{\partial^\ir}{\partial x_1^{\alpha_1} \cdots \partial x_n^{\alpha_n} }D_{L x_j}^s u \bigg\|_{L^1(Q)} +\varepsilon_n,\qquad
 \sum_{k=1}^n\alpha_k=\ir\]
 where $\varepsilon_n\to0$ as $n\to+\infty$. Here $D_{L x_j}^s$ denotes the RL derivative in the $x_j$ coordinate. 
 We discuss the case $j=1$, as the cases $j=2,\cdots,n$ are analogous.\\\\
 Similarly to \eqref{1st form}, by integration by parts and using the zero boundary condition of $u$ (and its derivatives, up to the $\ir$-th order), we get
  \begin{align*}
  &\frac{\partial^\ir}{\partial x_1^{\alpha_1} \cdots \partial x_n^{\alpha_n} }D_{L x_1}^s u(x_1,\cdots,x_n) \\
 &= \frac{1}{\Gamma(1-s)}\int_0^{x_1} \bigg[\frac{\partial^\ir}{\partial x_1^{\alpha_1} \cdots \partial x_n^{\alpha_n} }  \frac{\partial }{\partial x_1} u(t,x_2,\cdots,x_n) \bigg](x_1-t)^{-s} dt\\
 &=\frac{1}{\Gamma(1-s)} \bigg[  \bigg(\frac{\partial^\ir}{\partial x_1^{\alpha_1} \cdots \partial x_n^{\alpha_n} }  \frac{\partial }{\partial x_1} u\bigg)\chi_{(0,1)} * |\cdot|^{-s}\chi_{(0,1)} \bigg](x_1,\cdots,x_n). 
  \end{align*}
Thus a sequence $u_n\subseteq C_c^\infty(Q)$ such that  $u_n\to u$ strongly in $H_0^{\ir+1}(Q)$ gives
\begin{align*}
& \bigg\|\frac{\partial^\ir}{\partial x_1^{\alpha_1} \cdots \partial x_n^{\alpha_n} }D_{L x_1}^s u_n \bigg\|_{L^1(Q)}\\
 &\le \bigg\|\frac{\partial^\ir}{\partial x_1^{\alpha_1} \cdots \partial x_n^{\alpha_n} }D_{L x_1}^s u \bigg\|_{L^1(Q)}+
 \bigg\|\frac{\partial^\ir}{\partial x_1^{\alpha_1} \cdots \partial x_n^{\alpha_n} }D_{L x_1}^s (u-u_n) \bigg\|_{L^1(Q)},
\end{align*}
 where
 \begin{align*}
 &\bigg\|\frac{\partial^\ir}{\partial x_1^{\alpha_1} \cdots \partial x_n^{\alpha_n} } D_{L x_1}^s (u-u_n) \bigg\|_{L^1(Q)} \\
 &=
 \frac{1}{\Gamma(1-s)} \bigg\|  \bigg(\frac{\partial^\ir}{\partial x_1^{\alpha_1} \cdots \partial x_n^{\alpha_n} }  \frac{\partial }{\partial x_1} (u-u_n)\bigg)\chi_{(0,1)} * |\cdot|^{-s}\chi_{(0,1)} \bigg\|_{L^1(Q)}\\
 &\le\frac{1}{\Gamma(1-s)} \bigg\|  \bigg(\frac{\partial^\ir}{\partial x_1^{\alpha_1} \cdots \partial x_n^{\alpha_n} }  \frac{\partial }{\partial x_1} (u-u_n)\bigg)\chi_{(0,1)}  \bigg\|_{L^1(Q)}\cdot
 \bigg\| |\cdot|^{-s}\chi_{(0,1)} \bigg\|_{L^1(Q)} \to 0.
 \end{align*}
Thus any sequence $u_n\subseteq C_c^\infty(Q)$ such that  $u_n\to u$ strongly in $H_0^{\ir+1}(Q)$ is acceptable.\\\\
Now we turn our attention to the case of $p$-norm, with $p>1$. Like
in the previous case (where we had $p=2$), we again discuss first the case with
 $s\in (0,1)$, in the generic $n$ dimensional case, and only with the left-sided RL derivative 
 (since the proof for the right-sided RL derivative is analogous). \\\\
 Recall that
\[TV_\ellp^s(u):= \sup\bigg\{\int_Q u \divg^s \varphi dx:\varphi \in C_c^\infty(Q),  |\varphi|_\ellp^\ast \le 1\bigg\}.\]
%Since in the above definition there is a supremum on the right-hand sided, 
Like in the case of $p=2$, where we had the optimal $\varphi=- {D_L^s u}/{|D_L^s u|}$ if we had sufficient regularity. For $p>1$, we first see what
the optimal $\varphi$ should be, if $u$ had sufficient regularity. 
We first prove
\begin{equation}
 \int_Q D_{L x_j}^s u \cdot \varphi dx \le \int_Q|D_{L x_j}^s u(x)|_{{\frac{p}{p-1}}} dx,\qquad
 j=1,\cdots,n.
 \label{p norm phi}
\end{equation}
We write
\[ \int_Q D_{L x_j}^s u \cdot \varphi dx = \sum_{i=1}^n \int_Q (D_{L x_j}^s u)_i(x) \varphi_i(x) dx,  \]
and for each coordinate, and fixed $x$, we aim to find the optimal $\varphi_i$. Let
$a_i:=(D_{L x_j}^s u)_i(x)$, $b_i:=\varphi_i(x)$, so our aim is to find
\begin{equation}
\max \big\{ \sum_{i=1}^n a_i b_i :\sum_{i=1}^n |b_i|^p \le 1\Big\}, \label{max1}
\end{equation}
which easily reduces to
\[\max \big\{ \sum_{i=1}^n a_i b_i :\sum_{i=1}^n |b_i|^p = 1\Big\}. \]
Let
\[F(b_1,\cdots,b_n,\lambda):= \sum_{i=1}^n a_i b_i -\lambda\Big( \sum_{i=1}^n |b_i|^p - 1 \Big),\]
and
\[\frac{\partial F}{\partial \lambda} = 1-\sum_{i=1}^n |b_i|^p =0,
\qquad
\frac{\partial F}{\partial b_i} = a_i -\lambda p b_i|b_i|^{p-2}=0,\quad i=1,\cdots,n, \]
which gives $b_i|b_i|^{p-2}=\dfrac{a_i}{ \lambda p} $, with $\lambda$ 
to be determined by condition $\sum_{i=1}^n |b_i|^p = 1$. Direct computation gives
\begin{align*}
1=\sum_{i=1}^n |b_i|^p &=  \sum_{i=1}^n \bigg(\frac{|a_i|}{ \lambda p}\bigg)^{\frac{p}{p-1}}
\Longrightarrow\lambda^{\frac{p}{p-1}} =\frac1p \bigg( \sum_{i=1}^n |a_i|^{\frac{p}{p-1}}\bigg) ^{\frac{p-1}p}\\
&\Longrightarrow b_i|b_i|^{p-2}=|a_i|\bigg( \sum_{i=1}^n |a_i|^{\frac{p}{p-1}}\bigg) ^{-\frac{p-1}p}.
\end{align*}
Note that $\sum_{i=1}^n a_i b_i$ is maximum when $a_i$ and $b_i$ have the same sign. Thus
\[b_i  = a_i |a_i|^{\frac1{p-1}-1}\bigg( \sum_{i=1}^n |a_i|^{\frac{p}{p-1}}\bigg) ^{-\frac1p},\]
which gives
\[\sum_{i=1}^n a_i b_i = \frac{\sum_{i=1}^n |a_i|^{\frac{p}{p-1}}}{( \sum_{i=1}^n |a_i|^{\frac{p}{p-1}})^{1/p}}
=\Big( \sum_{i=1}^n |a_i|^{\frac{p}{p-1}} \Big)^{\frac{p-1}p},\]
hence
\[\sum_{i=1}^n(D_{L x_j}^s u)_i(x) \varphi_i(x) \le 
\Big( \sum_{i=1}^n |(D_{L x_j}^s u)_i(x)|^{\frac{p}{p-1}} \Big)^{\frac{p-1}p} = |D_{L x_j}^s u(x)|_{\frac{p}{p-1}},\]
and
\begin{align}
&TV^s_\ellp(u) \le \sum_{j=1}^n\int_Q|D_{L x_j}^s u(x)|_{\frac{p}{p-1}} dx = \|D_{L}^s u\|_{L_p^1(Q)}, \\
 &L_p^1(Q):=\bigg\{f:Q\to \mathbb{R}: \int_Q|f|_{\frac{p}{p-1}} dx<+\infty \bigg\},
\label{Lp case}
\end{align}
which proves \eqref{p norm phi}.
%where 
%\[L_p^1(Q):=\bigg\{f:Q\to \mathbb{R}: \int_Q|f|_{\frac{p}{p-1}} dx<+\infty\bigg\}. \] 
Due to the equivalence of all norms $|\cdot|_{\frac{p}{p-1}}$, the space $L_p^1(Q)=L^1(Q)$. Then, since
we assumed $u\in H_0^1(Q)$, similarly to the case $p=2$ we get that
\begin{equation}
 TV^s_p(u) = \|D_{L}^s u\|_{L_p^1(Q)}. \label{equality 1}
 \end{equation}
Note our arguments can be extended to more general $r=\ir +s$, $s\in (0,1)$: in this case
we would have, by integration by parts,
\begin{align*}
TV_\ellp^r(u) &= 
\sup\bigg\{\int_Q u \divg^s \varphi dx:\varphi \in C_c^\infty(Q),  |\varphi|_p \le 1\bigg\}
\\
&=\sup\bigg\{\int_Q D^\ir(D_L^s u) \cdot \varphi dx:\varphi \in C_c^\infty(Q),  |\varphi|_p \le 1\bigg\},
\end{align*}
and all the computations from \eqref{max1} to \eqref{equality 1} can be repeated, now
with the roles of $a_i$ and $b_i$ replaced by
\[a_{\alpha_1,\cdots,\alpha_n,i}:=\frac{\partial^\ir}{\partial x_1^{\alpha_1} \cdots \partial x_n^{\alpha_n} } D_{Lx_i}^s u(x),
\qquad b_{\alpha_1,\cdots,\alpha_n,i}:=\varphi_{\alpha_1,\cdots,\alpha_n,i}(x),\]
which at the end will give the analogue of \eqref{equality 1} for the case $r>1$, i.e.
\[  TV^r_{\ell^p}(u) = \sum_{i=1}^n \sum_{\substack{\alpha_1,\cdots,\alpha_n \ge 0 \\ \alpha_1+\cdots+\alpha_n =\ir} }
\bigg\|\frac{\partial^\ir}{\partial x_1^{\alpha_1} \cdots \partial x_n^{\alpha_n} } D_{Lx_i}^s u\bigg\|_{L_p^1(Q)}\]
Now we discuss the case $p=1$, where the previous computations would need additional care due to the presence
of $\frac{p}{p-1}$ in several places. However, the discussion will be quite similar to the previous
case. For clarity purposes, we start by discussing the case $r\in (0,1)$. 
Note that, quite similarly to the case $p>1$, we choose a fixed $x$, and let 
$a_k$ be the components of $D_L^s u(x)$, and $b_k$ be the components of $\varphi(x)$. Thus we have to find
\[\max \sum_{k=1}^n a_kb_k \qquad \text{subject to } \qquad \sum_{k=1}^n |b_k|=1.\]
Let 
\[F=F(b_1,\cdots,b_n,\lambda):= \sum_{k=1}^n a_kb_k -\lambda\Big(\sum_{k=1}^n |b_k|-1\Big),\]
and 
\[\frac{\partial F}{\partial b_k} = a_k-\lambda \frac{b_k}{|b_k|},\quad k=1,\cdots,n,\qquad 
\frac{\partial F}{\partial \lambda} = 1-\sum_{i=1}^n |b_i|^p =0.\]
Unless $a_1=\cdots=a_n=a$, which easily gives 
\[\max \bigg\{\sum_{k=1}^n a_kb_k : \sum_{k=1}^n |b_k|=1 \bigg\} = |a|, \]
we note that it is not possible to get
$\frac{\partial F}{\partial b_k} =0$ for all $k$. Hence no critical point is present, and the maximum
value is attained on the boundary of the set $\{(b_1,\cdots,b_n):\sum_{k=1}^n |b_k|=1\}$. Such boundary is easily
shown to be the set
\[\{b_1,\cdots,b_n:b_j=\pm1 \text{ for an index } j, \text{ and } b_k=0 \text{ for all }k\neq j\},\]
it follows immediately that
\[\max \bigg\{\sum_{k=1}^n a_kb_k : \sum_{k=1}^n |b_k|=1 \bigg\} = \sup_{k} |a_k|. \]
Thus, by integrating on $Q$, we have
\[\int_Q D_L^s u \varphi dx \le  \int_Q |D_L^s u|_\infty dx, \]
which is the analogue of \eqref{Lp case}, and similarly, due to the density of $C_c^\infty(Q)$ functions in $L^p(Q)$,
we conclude again that
\[\int_Q D_L^s u \varphi dx =  \int_Q |D_L^s u|_\infty dx, \]
which is the analogue of \eqref{equality 1}. 
The case of $r>1$ is analogous: we repeat the same arguments, with $a_k $ and $b_k$ replaced by
$a_{\alpha_1,\cdots,\alpha_n,i}$ (representing the components of $D^\ir D^s_L u(x)$) and $b_{\alpha_1,\cdots,\alpha_n,i}$
(representing the components of $\varphi$).
This concludes the discussion of the case $p=1$.\\\\
All the subsequent arguments from the case $p=2$ can be applied straightforwardly, component-wise:
given a sequence $u_n\subseteq C_c^\infty(Q)$ such that $u_n\to u$ strongly in $H_0^{\ir+1}(Q)$,
  we have
\begin{align*}
& \bigg\|\frac{\partial^\ir}{\partial x_1^{\alpha_1} \cdots \partial x_n^{\alpha_n} }D_{L x_1}^s u_n \bigg\|_{L_p^1(Q)}\\
 &\le \bigg\|\frac{\partial^\ir}{\partial x_1^{\alpha_1} \cdots \partial x_n^{\alpha_n} }D_{L x_1}^s u \bigg\|_{L_p^1(Q)}+
 \bigg\|\frac{\partial^\ir}{\partial x_1^{\alpha_1} \cdots \partial x_n^{\alpha_n} }D_{L x_1}^s (u-u_n) \bigg\|_{L_p^1(Q)},
\end{align*}
 where, due to the equivalence between the norms $ \| \cdot \|_{L_p^1(Q)}$ and $ \| \cdot \|_{L^1(Q)}$,
 \begin{align*}
& \bigg\|\frac{\partial^\ir}{\partial x_1^{\alpha_1} \cdots \partial x_n^{\alpha_n} } D_{L x_1}^s (u-u_n) \bigg\|_{L_p^1(Q)} \\
 &=
 \frac{1}{\Gamma(1-s)} \bigg\|  \bigg(\frac{\partial^\ir}{\partial x_1^{\alpha_1} \cdots \partial x_n^{\alpha_n} }  \frac{\partial }{\partial x_1} (u-u_n)\bigg)\chi_{(0,1)} * |\cdot|^{-s}\chi_{(0,1)} \bigg\|_{L_p^1(Q)}\\
 &\le\frac{1}{\Gamma(1-s)} \bigg\|  \bigg(\frac{\partial^\ir}{\partial x_1^{\alpha_1} \cdots \partial x_n^{\alpha_n} }  \frac{\partial }{\partial x_1} (u-u_n)\bigg)\chi_{(0,1)}  \bigg\|_{L_p^1(Q)}
 \bigg\| |\cdot|^{-s}\chi_{(0,1)} \bigg\|_{L_p^1(Q)} \to 0.
 \end{align*}
 The proof is thus complete.
\end{proof}

\begin{corollary}
Let $r\in\R^+$ be given. Then we have
\be
SV^r(Q)\cap H_0^{\ir+1}(Q)=BV^r(Q)\cap H_0^{\ir+1}(Q).
\ee
\end{corollary}
\begin{proof}
This is a direct result from Theorem \ref{smooth_approx_bdy_condition} and Definition \ref{strict_conv_BVr}.
\end{proof}

\subsection{Center-sided operations}
\begin{define}
as well as the \emph{central-sided Riemann-Liouville derivative} (of order $r\in\R^+$) 
\be
d^{r}w(x) = \frac12 \fmp{d^{r}_{L}w(x)+(-1)^{\ir+1}d^{r}_{R}w(x)}.
\ee
\end{define}

\subsection{Upcoming works}
As mentioned in introduction, in our follow up work \cite{panxinyang2018realordertraining} we shall concentrate on analyzing the training scheme \eqref{S_scheme_L1_intro}-\eqref{S_scheme_L2_intro} and provide extension numerical simulations. Moreover, we shall introduce a finite approximation of such scheme which will help on numerical realization of Level 1 problem in \eqref{S_scheme_L1_intro}.\\\\
We also propose to further extend our definition of fractional order total variation $TV^r$, by considering
 the generalized fractional derivative introduced in \cite{doi:10.1137/17M1160318}, and study the relevant properties.

\section*{Acknowledgements}
The work of Pan Liu has been supported by the Centre of Mathematical Imaging and Healthcare and funded by the Grant ``EPSRC Centre for Mathematical and Statistical Analysis of Multimodal Clinical Imaging" with No. EP/N014588/1.
Xin Yang Lu acknowledges the support of NSERC Grant 
{\em ``Regularity of minimizers and pattern formation in geometric minimization problems''},
and of the Startup funding, and Research Development Funding
of Lakehead University.

\bibliographystyle{abbrv}
\bibliography{RTV_theory_arXiv}{}

\end{document}